\documentclass[11pt]{article}  

\usepackage{a4wide}

\usepackage{graphicx} 
\usepackage{amssymb}
\usepackage{amsmath}
\usepackage{amsthm}
\usepackage{amsfonts}
\usepackage{color}
\usepackage{bm}

\usepackage[latin1]{inputenc}
\usepackage[driverfallback=dvipdfm]{hyperref}

\newtheorem{theorem}{Theorem}[section]
\newtheorem{proposition}[theorem]{Proposition}

\newtheorem{lemma}[theorem]{Lemma}
\newtheorem{definition}[theorem]{Definition}
\newtheorem{remark}[theorem]{Remark}

\newtheorem{example}[theorem]{Example}

\newcommand{\R}{\mathbb R}
\newcommand{\RN}{\mathbb R^N}
\newcommand{\SN}{\mathbb{S}^{N-1}}

\newcommand{\I}{\mathcal{I}}

\title{Liouville type results for the fractional truncated Laplacians in a half-space}
\author{Giulio Galise and Hitoshi Ishii
}
\date{}

\begin{document}

\maketitle

\begin{abstract}
\noindent
Existence issues of viscosity supersolutions in the half-space $\R^N_+$, for a class of fully 
nonlinear integral equations involving the  fractional truncated Laplacians and a power-like nonlinearity in the unknown function, are addressed in this paper, the aim being to obtain  estimates on  the threshold exponents separating the existence from the nonexistence regimes.
\end{abstract}

\vspace{0.5cm}

{\small

}

\tableofcontents

\section{Introduction}

This paper is concerned with the existence/nonexistence of nontrivial viscosity supersolutions to 
\begin{equation}\label{2912eq1}
\left\{\begin{array}{rl}
\I_k^\pm u(x)+u^p(x)=0 & \text{in $\mathbb R^N_+$}\\
u=0 & \text{in $\RN\backslash\mathbb R^N_+$,}
\end{array}\right.
\end{equation} 
where $\mathbb R^N_+=\left\{x=(x',x_N)\in\mathbb R^{N-1}\times\mathbb R\,:\;x_N>0\right\}$, $N\geq2$ and $p$ is a real exponent.\\
Problems \eqref{2912eq1} involve  the nonlinear integral operators $\I_k^\pm$ that are extremal among linear operators with one dimensional fractional diffusion. They are defined as follows: given $\xi\in\RN$ with $|\xi|=1$, let
\begin{equation}\label{eq1}
\I_\xi u(x)=C_s\int \limits_{0}^{+\infty}\frac{u(x+\tau\xi)+u(x-\tau\xi)-2u(x)}{\tau^{1+2s}}\,d\tau,
\end{equation}
where $s\in(0,1)$ and $u:\RN\mapsto \R$ is any  sufficiently smooth function for which the right hand side of \eqref{eq1} is well defined. The quantity $C_s$ is a positive normalizing constant for which the asymptotic $$\I_\xi u(x)\to\left\langle D^2u(x)\xi,\xi\right\rangle \qquad\text{as $s\to1^-$}$$ is guaranteed. The operator $\I_\xi u$, which is linear in $u$, acts as the $2s$-fractional derivative of $u$ along the direction $\xi$.\\
For any integer $k\in\left\{1,\ldots,N\right\}$, we denote by ${\mathcal V}_k$ the family of $k$-dimensional orthonormal sets in $\RN$, i.e.
\begin{equation}\label{1710eq1}
{\mathcal V}_k=\left\{\left\{\xi_1,\ldots,\xi_k\right\}\in\underbrace{\mathbb R^N\times\ldots\times\mathbb R^N}_{k\;\text{times}}:\,\;\left\langle\xi_i,\xi_j\right\rangle=\delta_{ij}\;\;\;\forall i,j\in\left\{1,\ldots,k\right\}\right\}.
\end{equation}
The operators $\I_k^-u$ and $\I_k^+u$, named \emph{minimal and maximal $k$-th fractional truncated Laplacian} respectively, are defined  through the formulas
\begin{equation}\label{eq2}
\begin{split}
\I^-_ku(x)&=\inf_{\left\{\xi_i\right\}_{i=1}^k\in{\mathcal V}_k}\sum_{i=1}^k\I_{\xi_i}u(x)\\
\I^+_ku(x)&=\sup_{\left\{\xi_i\right\}_{i=1}^k\in{\mathcal V}_k}\sum_{i=1}^k\I_{\xi_i}u(x).
\end{split}
\end{equation}
They converge, as $s\to1^-$, to the sums of the smallest and largest $k$ eigenvalues of the Hessian matrix of $u$ which are degenerate elliptic operators arising in geometric contexts, see e.g. \cite{AmS,HMW,HL,Sha,Wu}. These limit PDEs constitute a class of Hamilton-Jacobi-Bellman 
equations of stochastic optimal control. We refer for this to, e.g.,\cite{FS}. 

\smallskip

The study of integral operators with diffusion supported along 
 lower dimensional sets of $\R^N$ generated some interest in the PDE community in recent years, also motivated by its connection with stable L\'evy process in Probability theory. As far as regularity issues for solutions associated to this kind of operators are concerned, we refer e.g. to \cite{BC,ROS} and the references therein.\\ 
The operator $\I^-_1$ leads to the notion of fractional first eigenvalue and it has been considered in \cite{DPQR} in order to extend the notion of convexity to the nonlocal setting. In \cite{BDPQR} the associated evolution problem has been also studied. In \cite{RR}, through min-max type formula involving one-dimensional fractional Laplacians,  the trace fractional Laplacian in  $\R^N$ was introduced, and its main properties were investigated. Further nonlocal approximations of the truncated 
Laplace
operators can be found in \cite{S} where related properties to the maximum principle were considered.

\medskip

Liouville type theorems play a fundamental role in the theory of integro-differential elliptic PDE. In addition 
 to 
its intrinsic interests concerning classifications of solutions, they are also employed to get existence results in bounded domains via blow-up procedure, see e.g \cite{GS}. The existence issue of solutions to inequalities posed in cone-like or exterior domains has been treated and generalized in different contexts. As far as  fully nonlinear uniformly elliptic operators involving a power-like dependence in $u$ are concerned, we refer e.g. to \cite{AS1,AS2,QS}, where the presence of \emph{critical exponents} that sharply characterize the range of $p$'s for which positive supersolutions exist or not rely on the existence of fundamental solutions and  various forms of the strong maximum principle and Harnack type inequalities. Moreover the existence of such fundamental solutions is obtained by means of an abstract topological fixed point theorem and compactness estimates in H\"older spaces. Such approach, which has been recently extended in \cite{NdPQ} to fully nonlinear integral operators, does not seem applicable to our strongly degenerate setting since elliptic estimates and Harnack inequalities are missing. Some  differences and counterexamples between the operators \eqref{eq2} we are interested in the present paper and more standard integral operators, like e.g. the fractional Laplacian $(-\Delta)^s$, can be found in \cite{BGS}. Nevertheless, using the following weak form of ellipticity 
\begin{equation}\label{well}
\text{$u\geq v$\; in\; $\R^N$\; and\; $u(x_0)=v(x_0)$}\quad\Longrightarrow\quad \I^\pm_ku(x_0)\geq\I^\pm_ku(x_0)
\end{equation}
and the sub/superadditivity properties of the operators $\I^\pm_k$, we still obtain some results concerning  the existence of nontrivial supersolutions of \eqref{2912eq1}  with respect to $p$.

Adopting the terminology used in \cite{GuSch}, condition \eqref{well} can be rephrased by saying that the operators $\I^\pm_k$ satisfy the \emph{global comparison property}.

\smallskip

Hencefoth we shall assume that the operators $\I^\pm_k$ act on functions belonging to the space $\mathcal S$, see Definition \ref{S} and Remark \ref{RemarkS}, in such a way $\I^\pm_ku(x_0)$ is well defined as soon as $u$ is twice differentiable around $x_0$.

\medskip

In order to describe our results, we first consider the case $p>0$.

\medskip
\noindent
 As far as the operator $\I^+_k$ is concerned, we shall use the \emph{ fundamental solution} $w_{\bar\gamma}(x)=|x|^{-\bar\gamma}$, whose existence was proved in \cite[Proposition 3.7]{BGT}, and the subadditivity property of $\I^+_k$, to build a strict subsolution $\psi$ of $\I^+_k\psi=0$, in the complement of a large ball, which behaves at infinity like $-\frac{\partial}{\partial x_N}w_{\bar\gamma}$.  The availability of such subsolution allows us to adapt, to our nonlocal and degenerate context, the rescaled-type test function argument used in \cite{L} for uniformly elliptic inequalities involving the minimal Pucci's operators and then applied also in \cite{BGL} to the case of truncated Laplacians.  
The number $\bar\gamma=\bar\gamma(k,s)\in(0,1)$, which is the unique exponent  for which $w_\gamma(x)=|x|^{-\gamma}$ is a radial solution of $\I^+_kw_\gamma(x)=0$ in $\RN\backslash\left\{0\right\}$ exists if, and only if, $k=1$ and $s\in(0,\frac12)$ or  $k\geq2$ and $s\in(0,1)$. As a consequence we obtain the following Liouville type theorem.

\begin{theorem}\label{th3}
Let $u\in LSC(\RN_+)\cap\mathcal S$ be a nonnegative viscosity supersolution of 
\begin{equation}\label{2210peq1}
\left\{\begin{array}{rl}
\I_k^+ u(x)+u^p(x)=0 & \text{in\; $\mathbb R^N_+$}\\
u=0 & \text{in\; $\RN\backslash\mathbb R^N_+$}
\end{array}\right.
\end{equation} 
where $s\in\left(0,\frac12\right)$ if $k=1$, $s\in(0,1)$ otherwise. If  $0<p<1+\frac{2s}{\bar{\gamma}+1}$, 
then $u(x)\equiv0$ in $\R^N_+$.
\end{theorem}

The next result complements Theorem \ref{th3} in the range of values of $k$ and $s$ for which $\bar\gamma$ does not exist. It relies on the fact that the function $w(x)=-|x|^{2s-1}$ is a solution, unbounded at infinity, of $\I^+_1w=0$ in $\R^N\backslash\left\{0\right\}$ for any $s\in\left[\frac12,1\right)$. 

\begin{theorem}\label{th4}
Let $u\in LSC(\RN_+)\cap\mathcal S$ be a nonnegative viscosity supersolution of \eqref{2210peq1} with $k=1$ and $s\in\left[\frac12,1\right)$.
If $0<p<\frac{1}{1-s}$, then $u(x)\equiv0$ in $\R^N_+$.
\end{theorem}

Differently from more standard elliptic integro-differential operators, new phenomena arises for $\I_k^-$ and $k<N$. In these case,  problem \eqref{2912eq1} admits in fact nontrivial supersolutions for any $p>0$. In the next theorem we prove the existence of supersolutions which are not monotone with respect the $x_N$-variable.

\begin{theorem}\label{th1}
Let $k<N$ be a positive integer. Then for any $p>0$ there exists a nonnegative viscosity supersolution $u\in C^{0,s}(\RN)\cap L^\infty(\RN)$  of 
$$
\I_k^- u(x)+u^p(x)=0 \quad\, \text{in\; $\mathbb R^N_+$}
$$
satisfying the following properties:
$$
u(x)=0\quad\forall x\in\RN\backslash\mathbb R^N_+\quad\quad\text{and}\qquad\limsup_{x_N\to+\infty}u(x',x_N)>0\quad\text{uniformly w.r.t. $x'\in\mathbb R^{N-1}$}.
$$
\end{theorem}

As far as $\I^-_N$ is concerned, we first point out that, differently from $\I^-_k$ with $k<N$, the operator $\I_N^-$ admits the fundamental solution $w_{\tilde\gamma}(x)=|x|^{-\tilde\gamma}$, where now $\tilde\gamma>1$ (if $N\geq3$) and converges to $N-2$ as $s\to1^-$, see \cite[Remark 4.9 and Lemma 6.2]{BGT} for details. \\
The nonexistence of positive (in $\R^N_+$) supersolutions to \eqref{2912eq1} in the sublinear case $0<p\leq1$ will be obtained in a similar way as for $\I^+_k$.

\begin{theorem}\label{INsublinear}
Let $u\in LSC(\RN_+)\cap\mathcal S$ be a nonnegative viscosity supersolution of 
\begin{equation*}
\left\{\begin{array}{rl}
\I_N^- u(x)+u^p(x)=0 & \text{in\; $\mathbb R^N_+$}\\
u=0 & \text{in\; $\RN\backslash\mathbb R^N_+$}.
\end{array}\right.
\end{equation*} 
If  $0<p\leq1$, 
then $u(x)\equiv 0$ in $\RN_+$.
\end{theorem}

The main difference between $\I_N^-$ and $\I^+_k$ occurs in the superlinear range $p>1$. Because of the lack of the subadditivity property for $\I_N^-$   we cannot infer that the function  $\psi(x)=-\frac{\partial w_{\tilde\gamma}}{\partial x_N}$, or a modification of it near the origin, is still a solution of $\I_N^-\psi>0$ in $\R^N_+$. Hence a corresponding Liouville theorem is missing in this case.
 Nevertheless, using the fact that $\I_N^-$ is superadditive, we shall prove that the range of $p$'s  for which \eqref{2912eq1} admits positive  supersolutions is in fact strictly larger than  $(1+\frac{2s}{\tilde\gamma},\,+\infty)$, which is the corresponding range of existence of  positive  entire supersolution , see \cite[Theorem 4.1]{BGT}.

\begin{theorem}\label{thIN}
There is a number $\gamma^+\in(\tilde\gamma,+\infty)$ such that for any $p>1+\frac{2s}{\gamma^+}$ the equation
$$
\I^-_Nu(x)+u^p(x)=0\quad\text{in}\quad\RN_+
$$
admits a nonnegative supersolution $u\in C(\RN)\cap L^\infty(\RN)$ satisfying the following properties:
$$
u(x)=0\quad\forall x\in\RN\backslash\RN_+\,,\quad u(x)>0\quad\forall x\in\RN_+\,,\quad\lim_{|x|\to+\infty}u(x)=0.
$$
\end{theorem}

\medskip

As far as the singular case $p\leq0$ is concerned, we first show that under suitable growth assumption at infinity for $u$, then \eqref{2912eq1} has no positive supersolution. 

\begin{theorem}\label{thmsingular}
Let $p\leq0$. Let $\mathcal I$ be any of $\I^\pm_k$ with $k=1,\ldots,N$. Then the problem
\begin{equation}\label{2810eq1}
\begin{cases}
u\in LSC(\R^N_+)\cap\mathcal S\\
u\geq0 & \text{in $\R^N\backslash\R^N_+$}\\
0<u(x)\leq C(1+|x|)^\alpha & \text{for $x\in\R^N_+$}\\
\mathcal Iu+u^p\leq0 & \text{in $\R^N_+$ (in the viscosity sense),}
\end{cases}
\end{equation}
where $C>0$ and $0<\alpha<\frac{2s}{1-p}$ are constants, has no solution.
\end{theorem}

It is noteworthy that for uniformly elliptic operators (see e.g. \cite[Theorem 5.1]{AS1}-\cite[Proposition 3,2]{M} and \cite[Theorem 1.3 and Corollary 1.1]{NdPQ} for the local and nonlocal case respectively), the nonexistence of positive supersolutions  for $-1\leq p\leq0$ is obtained without further restriction on the growth of $u$, contrary to our assumption in \eqref{2810eq1}. On the other hand, the corresponding proofs heavily rely on quantitative strong comparison principle and weak Harnack inequality, which are not at disposal for our degenerate operators. This partially motivate the use of additional assumptions, naturally arising in the rescaled test function method employed in the proof of Theorem \ref{thmsingular}. We also remark that the growth condition in \eqref{2810eq1} can be replaced with
\begin{equation}\label{2810eq1p}
u(x)=o(|x|^{\frac{2s}{1-p}})\;\quad\text{as $|x|\to\infty$.}
\end{equation}
Positive supersolutions growing at infinity with order ${\frac{2s}{1-p}}$ appear as soon as $p<-1$,  at least for some operators $\mathcal I$ as expressed by the following theorem.

\begin{theorem}\label{th-singular}
Let $\I=\I^-_k$  for $1\leq k\leq N$ or  $\I=\I^+_N$. Then for any $p<-1$ there exists a positive supersolution $u\in C^{0,\frac{2s}{1-p}}(\R^N)\cap C^\infty(\R^N_+)$ of 
$$
\I u(x)+u^p(x)=0\quad\text{in\, $\RN_+$}
$$
such that $u(x)=0$ for any $x\in\R^N\backslash \R^N_+$.
\end{theorem}

\medskip

We end this introduction by pointing out that all the above results imply some explicit bounds on the critical exponents
\begin{equation}\label{9924eq1}
\begin{split}
p^\star\left(\I^\pm_k\right)&:=\inf\left\{p>1:\;\text{\eqref{2912eq1} has a nontrivial supersolution}\right\}\\
p_\star\left(\I^\pm_k\right)&:=\sup\left\{p<1:\;\text{\eqref{2912eq1} has a nontrivial supersolution}\right\}.
\end{split}
\end{equation}
The definitions \eqref{9924eq1}, borrowed from \cite{KLS}, relies on the fact that if  for some $p>1$ ($p<1$) there exists a nontrivial supersolution of \eqref{2912eq1}, then nontrivial supersolutions appear for any $q>p$ ($q<p$), see Propositions \ref{PropS}-\ref{PropS2}.\\ 
We recall that the operators $\I^+_k$ with $k\leq N$ and $\I^-_N$ satisfy the strong maximum principle, see e.g. \cite[Theorem 4.3]{BGS}, differently from $\I^-_k$ with $k<N$. Accordingly, in the definitions \eqref{9924eq1} for  $\I^+_k$ and $\I^-_N$ we can refer without loss of generality to positive supersolutions, while in the case $\I^-_k$, $k<N$, nontrivial supersolutions of \eqref{2912eq1} may vanish somewhere in $\R^N_+$ as proved in Theorem \ref{th1}. 

\smallskip
All the above results can be summarized in the following scheme

\medskip
\begin{center}
\begin{tabular}{|c|c|c|c|}
 \hline \textbf{Operator} & $\bm{p^\star}$ & $\bm{p_\star}$ + \eqref{2810eq1p} & $\bm{p_\star}$\\
 \hline $\bm{\I^-_k}$ with $k<N$ &1 & 1 & 1\\ 
\hline  $\bm{\I^-_N}$ & $\leq 1+\frac{2s}{\gamma^+}$ & $-\infty$ & $\in[-1\,,\,0]$ \\
\hline  $\bm{\I^+_k}$ with $k<N$  & $\displaystyle\geq\begin{cases}
\frac{1}{1-s} & \text{if $k=1$ and $s\in[\frac12,1)$}\\
	1+\frac{2s}{\bar\gamma+1} & \text{otherwise}
	\end{cases}$ & $-\infty$ & $\leq0$\\
\hline  $\bm{\I^+_N}$ & $\displaystyle\geq\begin{cases}
\frac{1}{1-s} & \text{if $k=1$ and $s\in[\frac12,1)$}\\
	1+\frac{2s}{\bar\gamma+1} & \text{otherwise}
	\end{cases}$ & $-\infty$ & $\in[-1\,,\,0]$\\
\hline
\end{tabular}
\end{center}

\smallskip

The explicit expression of $p^\star$ and $p_\star$ is thus achieved only for $\I^-_k$ and $k<N$, while in the other cases some indeterminacy occurs, leading to open problems that we list below:
\begin{enumerate}
	\item  obtain a lower bound for $p^\star(\I^-_N)$, which in turn depends on the availability of Liouville type results in the supercritical regime $p>1$;
	\item  give a upper bound for $p^\star(\I^+_k)$ by producing explicit examples of positive supersolutions. At the present stage we can only infer that
	$$
	p^\star\leq1+\frac{2s}{\bar\gamma},
	$$
	the right-hand side of the above inequality being  the critical exponent in the whole space $\R^N$, see \cite[Theorem 4.2]{BGT};
	\item   estimate from below $p_\star(\I^+_k)$ in the cases $k<N$, by obtaining explicit positive supersolutions in the singular regime $p<0$.
\end{enumerate}

The rest of this paper is organized as follows. In Section \ref{pr}, after presenting some notations and definitions, we prove some preliminary results. In Section \ref{NoEx} we present the proofs of all nonexistence results contained in the theorems stated in the Introduction, while Section \ref{Ex} is devoted to the existence part. Finally, the Appendix deals with some properties of the critical exponents defined  in \eqref{9924eq1}.

\bigskip
\noindent
\textbf{Notations.} Throughout the paper, $B_r(x)$ is the open ball of radius $r>0$ centered at $x\in\R^N$. If $x,y\in\R^N$, their scalar product is $\left\langle x,y\right\rangle=\sum_{i^1}^Nx_iy_i$ and $|x|$ stands for the Euclidean norm of $x$. For $x\in\R^N\backslash\left\{0\right\}$, we shall denote $\hat x=\frac{x}{|x|}$ and $x^\perp$ is any unit vector orthogonal to $x$.\\
The positive and negative parts of $t\in\R$ are $t_+=\max\left\{t,0\right\}$ and $t_-=\max\left\{-t,0\right\}$, respectively.\\
If $\Omega\subseteq\R^N$ is open, we shall denote by $LSC(\Omega)$, $USC(\Omega)$ and $C(\Omega)$ the spaces of lower semicontinuous, upper semicontinuous and continuous real functions on $\Omega$, respectively. If $\alpha\in(0,1]$, then $C^{0,\alpha}(\Omega)$ is the space of $\alpha$-H\"older continuous functions in $\Omega$.\\
The support of $u:\Omega\mapsto\R$ is denoted by ${\rm supp}\,u$.

\section{Some preliminaries}\label{pr}

This section is devoted to some technical results that will be employed in the proofs of the main results of the paper. We start by introducing a suitable function space in which  the notion of viscosity solution to \eqref{2912eq1} is settled.

\begin{definition}\label{S}
The set $\mathcal S$ is the space of all functions $u:\R^N\mapsto\R$ satisfying the following assumptions:
\begin{itemize}
	\item[(H1)] for any $x\in\R^N_+$ and any unit vector $\xi\in\R^N$, the map
	$$
	\tau\in\R\mapsto u(x+\tau\xi)
	$$
	is measurable;
	\item[(H2)] there exists $\alpha\in[0,2s)$ and $C>0$ such that 
	$$
	\left|u(x)\right|\leq C{(1+|x|)}^\alpha\quad\forall x\in\R^N.
	$$
\end{itemize}
\end{definition}

\begin{remark}\label{RemarkS}
{\rm
The set $\mathcal S$ give a sufficient condition for which if $u\in\mathcal S$ is twice differentiable in a neighborhood of $x_0\in\R^N_+$ then the quantities $\I^\pm_ku(x_0)$ are well defined.\\
We point out that $\I^\pm_ku(x_0)$ may be finite even for some singular or unbounded and oscillating function $u:\R^N\mapsto\R$.
}
\end{remark}

\begin{definition}\label{defviscsol}
In the following definition we denote by $\mathcal I$ any of the operators $\I^\pm_k$.\\
Let $f\in C(\R^N_+\times\R)$. A function $u\in \mathcal S\cap LSC(\R^N_+)$ ($u\in \mathcal S\cap USC(\R^N_+)$) is a viscosity supersolution (subsolution) of
$$
\mathcal Iu+f(x,u)=0\quad\text{in $\R^N_+$}
$$
if for any $x_0\in\R^N_+$ and every $\varphi\in C^2(\overline B_r(x_0))$, $r>0$, such that $B_r(x_0)\subseteq\R^N_+$ and 
$$
\varphi(x_0)=u(x_0)\;,\;\;\varphi(x)\leq u(x)\;\;(\varphi(x)\geq u(x) )\quad\forall x\in B_r(x_0),
$$
then if we let
$$
\psi(x)=\begin{cases}
\varphi(x) & \text{in $B_r(x_0)$}\\
u(x) & \text{in $\R^N\backslash B_r(x_0)$}
\end{cases}
$$
we have
\begin{equation}\label{2210eq1}
\mathcal I\psi(x_0)+f(x_0,u(x_0))\leq 0 \quad \left(\mathcal Iu(x_0)+f(x_0,u(x_0))\geq0)\right).
\end{equation}
A function $u\in\mathcal  S\cap LSC(\R^N_+)$ ($u\in \mathcal S\cap USC(\R^N_+)$) is a viscosity supersolution (subsolution) of\eqref{2912eq1} if it satisfies \eqref{2210eq1} with $f(x,u)=u^p$ and $u\geq0$ ($u\leq0$) in $\R^N\backslash\R^N_+$. 
\end{definition}

We remark that the above definition does not cover \eqref{2912eq1}, with $p\leq 0$,  
but, for such a problem, we will argue only the viscosity properties of those functions $u$ which are positive in $\R^N_+$.   Another remark is that, although good stability properties may be lost, 
the above definition makes sense even if $f : \R^N_+\times \R \to \R$ is not continuous. 

\begin{proposition}\label{w}
Let $s\in(0,1)$ and let $\mathcal I$ be either $\I^+_1,\ldots,\I^+_N$ or $\I^-_N$. Then there exists a positive radial function $w\in C^3(\R^N)$, with bounded derivatives up to order 3, such that 
\begin{equation}\label{0111eq1}
\begin{cases}
r\mapsto w(r)\;\; \text{is decreasing}\\
\I w(x)>0\;\;\text{if $|x|\geq1$.}
\end{cases}
\end{equation}
\end{proposition}
\begin{proof}
The construction of $w$ is done by modifying the function $|x|^{-\gamma}$ near the origin, where $\gamma>0$ has to be chosen  sufficiently large. For $\gamma>0$, let $\varphi_\gamma(r)=r^{-\gamma}$ and set
\begin{equation*}
w_\gamma(x)=\begin{cases}
\tilde f(|x|^2) & \text{if $|x|^2<\frac12$}\\
|x|^{-\gamma} & \text{if $|x|^2\geq\frac12$\,,}
\end{cases}
\end{equation*}
where $\tilde f$ is such that the graph of  $\tilde f''(r)$ is the tangent line to ${\left(\varphi_{\frac\gamma2}(r)\right)}''$ at $r=\frac12$. In addition we require that
$$
\tilde f\left(\frac12\right)=\varphi_{\frac\gamma2}\left(\frac12\right),\quad \tilde f'\left(\frac12\right)=\varphi'_{\frac\gamma2}\left(\frac12\right).
$$
A straightforward computation yields
\begin{equation*}
\tilde f(r)=\frac{\varphi'''_{\frac\gamma2}(\frac12)}{6}(r-\frac12)^3+\frac{\varphi''_{\frac\gamma2}(\frac12)}{2}(r-\frac12)^2+\varphi'_{\frac\gamma2}(\frac12)(r-\frac12)+\varphi_{\frac\gamma2}(\frac12).
\end{equation*}
By construction $w_\gamma\in C^3(\mathbb R^N)$ and all its derivatives, up to order 3, are bounded on $\RN$. Moreover $w$ is radially decreasing and
\begin{equation*}
w_\gamma(x)=\tilde g(|x|^2)
\end{equation*} with both $\tilde g$ and $\tilde g''$ positive and convex functions in $[0,+\infty)$. 

\smallskip

We consider  first the case $\I=\I^-_N$. According to \cite[Theorem 3.4-(iii)]{BGT}, for any $x\in \R^N$
\begin{equation}\label{301224eq1}
\I^-_Nw_\gamma(x)=N\I_{\xi^*}w_\gamma(x)\,,
\end{equation}
where $\xi^*\in\R^N$ is a unit vector such that $\left\langle \hat x,\xi^*\right\rangle=\frac{1}{\sqrt{N}}.$ Moreover, for $|x|\geq1$, we also have
$$
\left|x\pm\tau\xi^*\right|^2\geq\frac12\qquad\forall\tau\geq0
$$
and thus
\begin{equation}\label{301224eq2}
w_\gamma(x\pm\tau\xi^*)=\left|x\pm\tau\xi^*\right|^{-\gamma}\qquad\forall\tau\geq0.
\end{equation}
By \eqref{301224eq1}-\eqref{301224eq2} 
\begin{equation}\label{301224eq3}
\I^-_Nw_\gamma(x)=NC_sc(\gamma)|x|^{-\gamma-2s} \qquad\text{for any $|x|\geq1$,}
\end{equation}
where the function $c(\gamma)$, defined by
$$
c(\gamma)=\int\limits_0^{+\infty}\frac{\left(1+\tau^2+\frac{2}{\sqrt N}\tau\right)^{-\frac\gamma{2}}
+\left(1+\tau^2-\frac{2}{\sqrt N}\tau\right)^{-\frac{\gamma}2}-2}{\tau^{1+2s}}\,d\tau,
$$
has a unique zero $\tilde\gamma=\tilde\gamma(N,s)>0$ such that $c(\gamma)<0$ for $\gamma\in(0,\tilde\gamma)$ and $c(\gamma)>0$ for $\gamma\in(\tilde\gamma,+\infty)$, see \cite[Lemma 4.8]{BGT}. Then, by \eqref{301224eq3}, if  we pick any $\gamma>\tilde\gamma$ we obtain that
\begin{equation}\label{3110eq1}
\I^-_Nw_\gamma(x)>0\qquad\text{for any $|x|\geq1$.}
\end{equation}
If instead $\I=\I^+_k$, then by means of \cite[Lemma 3.3 and Theorem 3.4-(i)]{BGT}, we have
$$
\I^+_kw_\gamma(x)=I(|x|,1)+(k-1)I(|x|,0),
$$
where, for $\theta\in [0,1]$, 
$$
I(|x|,\theta)=C_s|x|^{-2s}\int\limits_0^{+\infty}\frac{\tilde g(|x|^2(1+\tau^2+2\tau\theta))+\tilde g(|x|^2(1+\tau^2-2\tau\theta))-2\tilde g(|x|^2)}{\tau^{1+2s}}\,d\tau.
$$
Note that, since $\tilde g'(r)<0$,
$$
I(|x|,0)=2C_s|x|^{-2s}\int\limits_0^{+\infty}\frac{\tilde g(|x|^2(1+\tau^2))-\tilde g(|x|^2)}{\tau^{1+2s}}\,d\tau<0.
$$
Hence
$$
\I^+_kw_\gamma(x)\geq I(|x|,1)+(N-1)I(|x|,0)=\I^+_Nw_\gamma(x)\geq\I^-_Nw_\gamma(x)
$$
and the result follows by \eqref{3110eq1}.
\end{proof}

\begin{proposition}\label{extended}
Let $s\in(0,1)$ and $\I$  be either $\I^+_1,\ldots,\I^+_N$ or $\I^-_N$. Then there exists a constant $c>0$ such that if $u\in LSC(B_{2R}(z))$ is, for some $R>0$ and $z\in\R^N$, a supersolution of 
$$
\I u=0\qquad\text{in $B_{2R}(x)$}
$$
and $u\geq0$ in $\R^N$, then
$$
\min_{\overline{B_R(z)}}u\geq c \min_{\overline{B_{R/2}(z)}}u.
$$
\end{proposition}
\begin{proof}
By the translation invariance, we may assume $z=0$. Let $w(x)=w(|x|)$ be the function produced by Propostion \ref{w} and satisfying \eqref{0111eq1}. Set
$$
\psi_R(x)=\psi_R(|x|)=w\left(\frac{2x}{R}\right).
$$
Then
$$
\I\psi_R\left(\frac{Rx}{2}\right)={\left(\frac{R}{2}\right)}^{-2s}\I w(x)>0\quad\;\text{if $|x|\geq1$},
$$
that is 
$$
\I\psi_R(x)>0\quad\;\text{if $|x|\geq\frac R2$}.
$$
Set 
$$
M_r=\min_{\overline{B_r(0)}} u\quad\;\text{for $r>0$}
$$
and define
$$
\phi(x)=\phi(|x|)=\frac{\psi_R(|x|)-\psi_R(2R)}{\psi_R(0)-\psi_R(2R)}\quad\;\text{for $x\in\R^N$}.
$$
Note that $r\mapsto\phi(r)$ is decreasing in $[0,+\infty)$, $\phi(0)=1$ and $\phi(r)\leq0$ for $r\geq 2R$.

If $M_{R/2}=0$, then we trivially have $M_R\geq c M_{R/2}$ for any $c>0$.

We consider the other case, so that $M_{R/2}>0$. Set $v(x)=M_{R/2}\phi(x)$ for $x\in\R^N$. It follows that if $|x|\leq R/2$, then
$$
u(x)\geq M_{R/2}=M_{R/2}\phi(0)\geq M_{R/2}\phi(x),
$$
and if $|x|\geq 2R$, then $\phi(x)\leq0$ and 
$$
u(x)\geq0\geq M_{R/2}\phi(x).
$$
We claim that 
$$
u\geq v\quad\text{in $\R^N$.}
$$
Suppose to the contrary that 
$$
\inf_{\R^N}(u-v)<0.
$$
Then there exists $x_R\in B_{2R}(0)\backslash\overline{B_{R/2}(0)}$ such that 
$$
(u-v)(x_R)=\min_{\R^N}(u-v).
$$
Since $u$ is a supersolution in $B_{2R}(0)$, we must have
$$
0\geq \I v(x_R)=\frac{M_{R/2}}{\psi_R(0)-\psi_{R}(2R)}\I\psi_R(x_R).
$$
But, since $|x_R|>R/2$, we have $\I\psi_R(x_R)>0$, which yields a contradiction. Thus, for $x\in\overline{B_{R}(0)}$, we have
$$
u(x)\geq v(x)=M_{R/2}\phi(x)\geq M_{R/2}\phi(R),
$$
which implies that 
$$
M_R\geq M_{R/2}\frac{\psi_R(R)-\psi_R(2R)}{\psi_R(0)-\psi_R(2R)}=M_{R/2}\frac{w(2)-w(4)}{w(0)-w(4)}.
$$
Thus, if we set
$$
c=\frac{w(2)-w(4)}{w(0)-w(4)}
$$
then $M_R\geq c M_{R/2}$.
\end{proof}

\begin{remark}\label{SMP}
{\rm
Proposition \ref{extended} can be used to give  a proof of the strong minimum principle \cite[Theorem 4.3 (ii)-(iii)]{BGS}. For instance, assume that $u\in LSC(\R^N)$ is a nonnegative (in $\R^N$) supersolution of $\I^-_Nu=0$ in an open and connected set $\Omega\subseteq\R^N$. We have either $u\equiv0$ in $\Omega$ or $u>0$ in $\Omega$.\\
To see this, set $Z=\left\{x\in\Omega:\;u(x)=0\right\}$ and $P=\left\{x\in\Omega:\;u(x)>0\right\}$. We need to show that either $Z=\emptyset$ or $P=\emptyset$. Suppose that both are nonempty sets. Since $u\in LSC(\R^N)$, $Z$ is closed in $\Omega$ and $P$ is open. If $P$ is closed relatively in $\Omega$, then the connectedness of $\Omega$ tells us that $P=\Omega$, which is the desired contradiction. Otherwise, $P$ is not relatively closed in $\Omega$. Here, note by definition that $Q\subseteq \Omega$ is relatively closed in $\Omega\; \Longleftrightarrow\; Q=F\cap\Omega$ for some closed $F$ in $\R^N\;\Longleftrightarrow\; Q=\overline Q\cap\Omega$, where $\overline Q$ is the closure of $Q\in \R^N$. Thus, by the current assumption, we have $P\neq\overline P\cap\Omega$. Since $P\subseteq\Omega$, this implies that there is a point $z\in\overline P\cap\Omega$ such that $z\notin P$. That is, $z\in\Omega$, $z\in\partial P$ and $z\in Z$. We may choose $R>0$ so that $B_{5R}(z)\subseteq\Omega$. Since $z\in\partial P$, there exists $y\in P\cap B_R(z)$. Note that $B_{4R}(y)\subseteq B_{5R}(x)\subseteq\Omega$. Since $Z$ is closed in $\Omega$ (i.e. $Z=\overline Z\cap\Omega$), which implies that $Z\cap B_{4R}(y)=\overline Z\cap B_{4R}(y)$, and $z\in B_R(y)\cap Z$, there exists a closest point $q\in Z$ from $y$. Since it is closest, if we set $r:=|q-y|$, then $r<R$ and $B_r(y)\subseteq P$. For any $\rho\in(0,r)$, since $\overline{ B_{\rho}(y)}\subseteq P$, we have $\min_{\overline{ B_{\rho}(y)}}u>0$. Note that $B_{4\rho}(y)\subseteq\Omega$. By Proposition \ref{extended}, we find that for some $c>0$,
 $$
\min_{\overline{B_{2\rho}(y)}}u\geq c\min_{\overline{B_{\rho}(y)}}u>0.
$$
We can choose $\rho\in(0,r)$ so that $2\rho>r$, which implies that $q\in B_{2\rho}(y)$ and, moreover, $u(q)>0$, a contradiction.
}
\end{remark}

\begin{lemma}\label{lem1}
Let $\xi\in\SN$ and let $v\in C^3(\RN)$ be such that all its derivatives, up to order 3, are bounded on $\RN$. If 
\begin{itemize}
	\item[(i)] $v$ is bounded 
\end{itemize}
or 
\begin{itemize}
	\item[(ii)] $v$ is unbounded and $s\in(\frac12,1)$
\end{itemize}
 then the function 
$$
y\mapsto\I_\xi v(y)
$$ 
is differentiable everywhere and 
\begin{equation}\label{3.4.2}
D\I_\xi v(y)=\I_\xi Dv(y).
\end{equation}
\end{lemma}
\begin{proof}
Write $\I_\xi v(y)$ as
\[
\I_{\xi}\psi(y)=C_s\int \limits_0^{+\infty} \frac{g(y,\tau)}{\tau^{1+2s}}\,d\tau,
\]
where 
\[
g(y,\tau)=v(y+\tau\xi)+v(y-\tau\xi)-2v(y). 
\]
Note $g\in C^3(\R^N\times\R)$, that the derivatives of $g$ up to of order 3 are bounded on $\R^{N+1}$
and that  
\[ \begin{aligned}
&g(y,0)=0, \quad
\\&
g_{y_j}(y,\tau)=v_{y_j}(y+\tau\xi)+v_{y_j}(y-\tau\xi)-2 v_{y_j}(y), 
\quad
g_{y_j}(y,0)=0,
\quad
\\& 
g_{\tau}(y,\tau)=\langle Dv(y+\tau\xi),\xi\rangle-\langle Dv(y-\tau\xi),\xi\rangle,
\quad
g_{\tau}(y,0)=0,
\quad
\\&
g_{y_j\tau}(y,\tau)=\langle D v_{y_j}(y+\tau\xi), \xi\rangle-\langle Dv_{y_j}(y-\tau\xi),\xi\rangle,
\quad
g_{y_j \tau }(y,0)=0. 
\end{aligned}
\]
Thus, we have
\begin{equation}\label{0601eq1}
g(y,\tau)=\frac 12 g_{\tau\tau}(y,\theta_1\tau)\tau^2,\qquad 
g_{y_j}(y,\tau)=\frac 12 g_{y_j\tau\tau}(y,\theta_2\tau)\tau^2
\end{equation}
for some $\theta_1,\theta_2\in(0,1)$. 

\medskip

\noindent
(i) Since $v$ is bounded, it follows that there is a constant $C>0$ such that 
\[
\max\{|g(y,\tau)|, |g_{y_j}(y,\tau)|\}\leq C\min\{\tau^2,1\}
\]
for all $j,y,\tau$. Hence there exists $C>0$ such that for any $\tau>0$
\[
\max\left\{
\left|\frac{g(y,\tau)}{\tau^{1+2s}}\right|, 
\left|\frac{Dg(y,\tau)}{\tau^{1+2s}}\right|
\right\}\leq C \min\{\tau^{1-2s}, \tau^{-1-2s}\}. 
\]
Since $\tau \mapsto \min\{\tau^{1-2s},\tau^{-1-2s}\}$ is integrable in $(0,\infty)$, 
we find that 
\[
D_{y_j}\int\limits_0^{+\infty} \frac{g(y,\tau)}{\tau^{1+2s}} d\tau=\int\limits_0^{+\infty} \frac{D_{y_j}g(y,\tau)}{\tau^{1+2s}}d\tau,  
\]
and
\[
D_{y_j}\I_{\xi} v(y)=\I_{\xi} D_{y_j}v(y).
\]
(ii) Also, we have
\begin{equation}\label{0601eq2}
|g(y,\tau)|\leq g_\tau(y,\theta_3\tau)\tau,\qquad 
|g_{y_j}(y,\tau)|\leq g_{y_j\tau}(y,\theta_4\tau)\tau
\end{equation}
for some $\theta_3,\theta_4\in(0,1)$. 

Using \eqref{0601eq1} and \eqref{0601eq2}, it follows that there is a constant $C>0$ such that 
\[
\max\{|g(y,\tau)|, |g_{y_j}(y,\tau)|\}\leq C\tau\min\{\tau,1\}
\]
for all $j,y,\tau$. Hence, 
\[
\max\left\{
\left|\frac{g(y,\tau)}{\tau^{1+2s}}\right|, 
\left|\frac{Dg(y,\tau)}{\tau^{1+2s}}\right|
\right\}\leq C \min\{\tau^{1-2s}, \tau^{-2s}\}. 
\]
Since by  assumption $s>\frac12$, then  $\tau \mapsto \min\{\tau^{1-2s},\tau^{-2s}\}$ is integrable in $(0,\infty)$ 
and once again we obtain  \eqref{3.4.2}. 
\end{proof}
\begin{remark}
{\rm  
We shall apply Lemma \ref{lem1} to functions  coinciding, for $|x|\geq1$, with $|x|^{-\gamma}$ and with $-|x|^\gamma$ for suitable $\gamma\in(0,1)$.}
\end{remark}

\begin{proposition}\label{0801prop1}
Let $\xi\in\mathbb S^{N-1}$ and let $\mu\in(0,s]$. The function 
$$
z(x)=(x_N)_+^\mu\qquad \text{for $x\in\RN$}
$$
is a supersolution of 
$$
\I_\xi z(x)=0\qquad\text{in $\RN_+$}.
$$
\end{proposition}
\begin{proof}
Fix $x\in\R^N_+$. If $\xi_N=0$, then 
\[
z(x+\tau \xi)+z(x-\tau\xi)-2z(x)=0 \ \ \text{ for } \tau\in\R,
\]
and $\I_{\xi}z(x)=0$. 

Assume that $\xi_N\neq0$. Since $\I_{-\xi}z(x)=\I_{\xi}z(x)$, 
we may assume that $\xi_N>0$. We have
$$
\I_\xi z(x)=C_s\xi_N^{2s}x_N^{\mu-2s}\int\limits_0^{+\infty} \frac{(1+\tau)^\mu+(1-\tau)_+^\mu-2}{\tau^{1+2s}}\,d\tau.
$$
Set
$$
c_{s,\mu}=\int\limits_0^{+\infty} \frac{(1+\tau)^\mu+(1-\tau)_+^\mu-2}{\tau^{1+2s}}\,d\tau.
$$
 Arguing as in \cite[Lemma 2.4]{BV}, but for general $\mu\in(0,2s)$ instead of $\mu=s$, we obtain that
$$
c_{s,\mu}=
\frac{\mu}{2s}\int\limits_0^{+\infty}\frac{(1+\tau)^{\mu-1}-(1+\tau)^{2s-\mu-1}}{\tau^{2s}}\,d\tau.
$$  
Hence the conclusion easily follows.
\end{proof}
\begin{remark}\label{T49-5}
{\rm
If $0<\mu<s$, then $c_{s,\mu}<0$ and for $x\in\RN_+$
$$
\I_{e_N}z(x)=C_sc_{s,\mu}x_N^{\mu-2s}=C_sc_{s,\mu}z^{1-\frac{2s}{\mu}}(x).
$$}
\end{remark}


\section{Liouville-type results}\label{NoEx}

This section is concerned with the nonexistence results of positive supersolutions expressed by Theorems \ref{th3}-\ref{th4}-\ref{INsublinear}-\ref{thmsingular}.

\medskip

Let $p\in\R$ and let $\I$ be denote any of $\I^\pm_k$ with $k=1,\ldots,N$. Assume that $u\in LSC(\R^N)\cap\mathcal S$ is a viscosity supersolution of
\begin{equation*}
\begin{array}{rl}
\I u(x)+u^p(x)=0 & \text{in\; $\mathbb R^N_+$}
\end{array}
\end{equation*}
such that
$$
u\geq0\quad\text{in $\R^N$},\qquad u>0\quad\text{in $\R^N_+$}.
$$
For $R>0$, set
$$
M_1(R)=\min_{\overline{B_R(5Re_N)}}u\qquad\text{and}\qquad M_2(R)=\min_{\overline{B_{2R}(5Re_N)}}u.
$$
Notice that $B_{5R}(5Re_N)\subset\R^N_+$.

\smallskip
Let $\eta:[0,+\infty)\mapsto\R$ be a smooth cut-off function such that 
$$
\eta(r)=1\quad\text{for $r\in[0,1]$},\qquad\eta(r)=0\quad\text{for $r\geq2$},\qquad \eta'(r)\leq0\quad\text{for any $r\geq0$.}
$$ 
Set 
$$
\zeta(x)=M_1(R)\eta\left(\frac xR-5e_N\right).
$$
For some positive constant $C_\eta$, we have
$$
\I\zeta(x)\geq- C_\eta M_1(R)R^{-2s}\quad\;\forall x\in\R^N.
$$
If $|x-5Re_N|\leq R$, then
$$
\zeta(x)= M_1(R)\leq u(x).
$$
If  $|x-5Re_N|\geq 2R$, then
$$
\zeta(x)=0\leq u(x).
$$
Noting that $$\zeta(x)=u(x)\;\quad\text{at some point $x\in\overline{B_R(5Re_N)}$},$$
we find that there is $z\in B_{2R}(5Re_N)$ such that 
$$
\min_{\R^N}(u-\zeta)=(u-\zeta)(z)\leq0.
$$
Since $z\in\R^N_+$, by the supersolution property for $u$, we have 
\begin{equation}\label{0111peq1}
u^p(z)\leq -\I \zeta(z)\leq C_\eta M_1(R)R^{-2s}.
\end{equation}

\noindent
\textbf{Case $p\leq0$.} Since $(u-\zeta)(z)\leq0$, we have
$$
u(z)\leq\zeta(z)\leq M_1(R),
$$
and hence
$$u^p(z)\geq M^p_1(R).$$
This combined with \eqref{0111peq1} yields
\begin{equation}\label{0111peq2}
1\leq C_\eta R^{-2s}M^{1-p}_1(R).
\end{equation}
\noindent
\textbf{Case $p>0$.} In this case we are concernend only with $\I=\I^-_N$ or $\I=\I^+_k$ for $k=1,\ldots,N$. In this way the conclusion of Proposition \ref{extended} is at our disposal. Since $z\in B_{2R}(5Re_N)$, we have
$$
u(z)\geq M_2(R)
$$
and this, together with \eqref{0111peq1}, yields
$$
M^p_2(R)\leq  C_\eta M_1(R)R^{-2s}.
$$
Since
$
B_{4R}(5Re_N)\subset\R^N_+
$, by Proposition \ref{extended} 
$$
M_2(R)\geq cM_1(R).
$$
Hence we obtain
\begin{equation}\label{0111peq3}
1\leq c^{-p}C_\eta R^{-2s}M^{1-p}_1(R).
\end{equation}

\subsection{The singular case $p<0$}
The proof of Theorem \ref{thmsingular} is a consenquence of the inequality \eqref{0111peq2}.

\begin{proof}[Proof of Theorem \ref{thmsingular}]
Suppose to the contrary that there is a function $u\in LSC(\R^N)\cap S$ satisfying \eqref{2810eq1}. Then, for $R\geq1$, one has 
$$
M_1(R)=\min_{\overline{B_R(5Re_N)}}u\leq C\min_{\overline{B_R(5Re_N)}}{(1+|x|)}^\alpha=C{(1+4R)}^\alpha\leq C{(5R)}^\alpha.
$$
Therefore, we find by \eqref{0111peq2} that 
$$
1\leq C_\eta C^{1-p}5^{\alpha(1-p)}R^{\alpha(1-p)-2s}.
$$
This give a contradiction for sufficiently large $R$ since $\alpha(1-p)-2s<0$.
\end{proof}

\subsection{The sublinear case $0<p\leq1$}

The proofs of Theorems \ref{th3}-\ref{th4}-\ref{INsublinear}, for $0<p\leq1$, follows from the following

\begin{proposition}\label{1novprop1}
Let $0<p\leq1$. For $\I=\I^-_N$ or $\I=\I^+_k$, $k=1,\ldots,N$, if $u\in LSC(\R^N_+)\cap\mathcal S$ is a nonnegative (in $\R^N$) viscosity supersolution of 
$$
\I u(x)+u^p(x)=0 \;\quad \text{in\; $\mathbb R^N_+$},
$$
then $u(x)\equiv0$ in $\R^N_+$.
\end{proposition}
\begin{proof}
Suppose to the contrary the $u\not\equiv0$ in $\R^N_+$. By the strong maximum principle, see \cite[Theorem 4.3 (ii)]{BGS} and Remark \ref{SMP}, we have $u>0$ in $\R^N_+$. Observe that for $R\geq1$,
$$
M_1(R)=\min_{\overline{B_R(5Re_N)}}u\leq C\min_{\overline{B_R(5Re_N)}}{(1+|x|)}^{\alpha}=C{(1+4R)}^\alpha\leq C{(5R)}^\alpha
$$
This together with \eqref{0111peq3} yields
$$
1\leq c^{-p}C_\eta C^{1-p}5^{\alpha(1-p)}R^{\alpha(1-p)-2s}
$$ 
which gives a contradiction when $R$ is sufficiently large since, by the assumption $u\in\mathcal S$, it turns out that $\alpha(1-p)-2s<0$.
\end{proof}
\begin{remark}
{\rm
Proposition \ref{1novprop1} continues to hold under the growth assumption $$|u(x)|\leq C(1+|x|)^\alpha$$ for some $C>0$ and $0\leq \alpha<\frac{2s}{1-p}$, which is more general then $u\in\mathcal S$, provided the operator $\I u$ is well defined.
}
\end{remark}


\subsection{Theorems \ref{th3}-\ref{th4} for $p>1$} \label{secLiouville}
In this section we consider the problem
\begin{equation}\label{eq21}
\I^+_ku(x)+u^p(x)\leq0\quad\text{in}\quad\RN_+\qquad\text{and}\qquad u\geq0\quad\text{in}\quad\RN.
\end{equation}

\noindent

\medskip

The proofs of the  Theorems \ref{th3} and \ref{th4}  rely on the construction of a (classical and strict) subsolution of the homogeneous equation $\I^+_ku=0$ in $\RN_+\backslash \overline{B_{R_0}(0)}$, for some large $R_0$. Such subsolution is then used to estimate from below any nontrivial solutions of \eqref{eq21}, which in turn implies an estimate from above of the quantity $M^{1-p}_1(R)$ in \eqref{0111peq3}.

\medskip
We start from Theorem \ref{th3}, hence we assume $s\in\left(0,\frac12\right)$ if $k=1$, while $s\in(0,1)$ if $k\geq2$. The constant $\bar\gamma\in(0,1)$, see \cite[Proposition 3.7]{BGT}, is the unique number vanishing the function 
\begin{equation}\label{811eq1}
c_k(\gamma)=\hat c(\gamma)+(k-1)c^\perp(\gamma)
\end{equation}
 where\footnote{Henceforth P.V. stands for the Cauchy Principal Value.}
\begin{equation}\label{811eq2}
\hat c(\gamma)=C_s\text{P.V.}\,\int\limits_{-\infty}^{+\infty}\frac{|1+\tau|^{-\gamma}-1}{|\tau|^{1+2s}}\,d\tau\;\quad\text{and}\;\quad
c^\perp(\gamma)=2C_s\int\limits_0^{+\infty}\frac{{(1+\tau^2)}^{-\frac\gamma2}-1}{\tau^{1+2s}}\,d\tau.
\end{equation}
The function $c_k(\gamma)$ is negative for $\gamma<\bar\gamma$ and positive $\gamma>\bar\gamma$.

\begin{proposition}\label{propSubsol}
There exist a constant $R_0>1$ and a function $\psi\in C^2(\RN)\cap L^\infty(\RN)$, depending on $k$ and $s$, such that 
\begin{equation}\label{3012eq1}
\I^+_k\psi(x)>0 \qquad \forall x\in\RN_+\backslash \overline{B_{R_0}(0)}
\end{equation}
\begin{equation}\label{3012eq2}
\psi(x)\geq\frac{x_N}{|x|^{\bar\gamma+2}} \;\quad\forall x\in\RN_+\backslash \overline{B_1(0)}\,\quad{and}\,\quad\lim_{\stackrel{|x|\to+\infty}{x_N\neq0}
}\frac{\psi(x)}{\frac{x_N}{|x|^{\bar\gamma+2}}}=1
\end{equation}
\begin{equation}\label{3012eq3}
\psi(x)>0\quad\text{if\; $x_N>0$},\;\;\;\psi(x)=0\quad\text{if\; $x_N=0$},\;\;\;\psi(x)<0 \quad\text{if\; $x_N<0$}.
\end{equation}
\end{proposition}
\begin{proof}
For $\gamma\in(0,1)$, let $\varphi_\gamma(r)=r^{-\gamma}$  and let $v_\gamma\in C^3(\mathbb\RN)\cap L^\infty(\RN)$ be the function defined by
\begin{equation}\label{3eq11}
v_\gamma(x)=\begin{cases}
\tilde f_\gamma(|x|^2) & \text{if $|x|\leq1$}\\
\varphi_\gamma(|x|) & \text{if $|x|>1$}
\end{cases}
\end{equation}
where for $r\geq0$
\begin{equation}\label{3012eq5}
\tilde f_\gamma(r)=\frac{\varphi'''_{\frac\gamma2}(1)}{6}(r-1)^3+\frac{\varphi''_{\frac\gamma2}(1)}{2}(r-1)^2+\varphi'_{\frac\gamma2}(1)(r-1)+\varphi_{\frac\gamma2}(1).
\end{equation}
By construction, $v_\gamma(x)=\tilde g(|x|^2)$ with both $\tilde g$ and $\tilde g''$ positive and convex functions in $[0,+\infty)$. The representation formula \cite[Theorem 3.4 (i)]{BGT} then applies to $v_\gamma$.

\smallskip

For  $\gamma>\bar\gamma$ fixed, set
\begin{equation}\label{0101eq1}
\psi(x)=-\frac{1}{\bar\gamma}\left(D_{x_N}v_{\bar\gamma}(x)+D_{x_N}v_{\gamma}(x)\right).
\end{equation}
 We claim that for $R_0=R_0(k,s)$ sufficiently large and for any $x\in\RN_+$ with $|x|>R_0$, then there exists a $k$-dimensional orthonormal frame $\left\{\xi_1,\ldots,\xi_k\right\}$, depending on $x$, such that 
\begin{equation}\label{claim}
\sum_{i=1}^k\I_{\xi_i}\psi(x)>0.
\end{equation}
This implies the desired result. Indeed, condition \eqref{3012eq1} is a consequence the maximality of the operator $\I^+_k$, since 
$$
\I^+_k\psi(x)\geq\sum_{i=1}^k\I_{\xi_i}\psi(x)>0 \qquad \forall x\in\RN_+\backslash \overline{B_{R_0}(0)}.
$$
Moreover, for any $|x|>1$, one has
\begin{equation}\label{3012eq4}
\psi(x)=\frac{x_N}{|x|^{\bar\gamma+2}}+\frac{\gamma}{\bar\gamma}\frac{x_N}{|x|^{\gamma+2}}
\end{equation}
and \eqref{3012eq2} follows from the fact that $\gamma>\bar\gamma>0$.\\
As far as \eqref{3012eq3} is concerned, it is clear that it holds for $|x|>1$ in view of \eqref{3012eq4}. If instead $|x|\leq1$, then using \eqref{3012eq5}, we have 
$$
\psi(x)=-\frac{2}{\bar\gamma}\left(\tilde f'_{\bar\gamma}(|x|^2)+\tilde f'_{\gamma}(|x|^2)\right)x_N
$$
with $$\tilde f'_{\bar\gamma}(|x|^2)+\tilde f'_{\gamma}(|x|^2)<0.$$
Hence condition $\eqref{3012eq3}$ is also satisfied for $|x|\leq1$.

\medskip
\noindent
In order to prove \eqref{claim}, we start by computing for $|x|>1$ and $n\in\mathbb N$ the quantity
\begin{equation}\label{3112eq1}
\I^+_kv_\gamma(x)-\I^+_kv_\gamma\left(x+\frac{e_N}{n}\right).
\end{equation}
Then, after multiplying \eqref{3112eq1} by $n$, we will send $n\to+\infty$.\\
Using  \cite[Theorem 3.4 (i)]{BGT} we have
\begin{equation}\label{3112eq2}
\begin{split}
\I^+_kv_\gamma(x)-\I^+_kv_\gamma\left(x+\frac{e_N}{n}\right)&=
\I_{\hat{x}}v_\gamma(x)+(k-1)\I_{x^\perp}v_\gamma(x)\\&\quad-\I_{\widehat{{x+\frac{e_N}{n}}}}v_\gamma\left(x+\frac{e_N}{n}\right)-(k-1)\I_{{(x+\frac{e_N}{n})}^\perp}v_\gamma\left(x+\frac{e_N}{n}\right).
\end{split}
\end{equation}
For $|x|>1$ and $x\in\R^N_+$ it turns out that 
$$
|x+\frac{e_N}{n}+\tau{(x+\frac{e_N}{n})}^\perp|>|x+\tau x^\perp|>1\qquad\forall \tau\in\R,\;n\in\mathbb N.
$$
Since $v_\gamma(x)=|x|^{-\gamma}$ for any $|x|>1$, then 
\begin{equation}\label{2eq3}
\begin{split}
\I_{x^\perp}v_\gamma(x)&=\I_{x^\perp}\phi_\gamma(x)\\
\I_{{(x+\frac{e_N}{n})}^\perp}v_\gamma\left(x+\frac{e_N}{n}\right)&=\I_{{(x+\frac{e_N}{n})}^\perp}\phi_\gamma\left(x+\frac{e_N}{n}\right),
\end{split}
\end{equation}
where $\phi_\gamma:\RN\backslash\left\{0\right\}\mapsto\R$ is the function $\phi_\gamma(x)=\varphi_\gamma(|x|)=|x|^{-\gamma}$.\\ On the other hand, a straightforward computation yields
\begin{equation}\label{2eq4}
\I_{\hat{x}}v_\gamma(x)=\I_{\hat{x}}\phi_\gamma(x)+C_s|x|^{-\gamma-2s}\int \limits_{-1-\frac{1}{|x|}}^{-1+\frac{1}{|x|}}\frac{|x|^\gamma\tilde f_\gamma\left(|x|^2(1+\tau)^2\right)-|1+\tau|^{-\gamma}}{|\tau|^{1+2s}}\,d\tau
\end{equation}
and similarly
\begin{equation}\label{2eq5}
\begin{split}
&\I_{\widehat{x+\frac{e_N}{n}}}v_\gamma\left(x+\frac{e_N}{n}\right)=\I_{\widehat{x+\frac{e_N}{n}}}\phi_\gamma\left(x+\frac{e_N}{n}\right)\\&+C_s|x+\frac{e_N}{n}|^{-\gamma-2s}\int \limits_{-1-\frac{1}{|x+\frac{e_N}{n}|}}^{-1+\frac{1}{|x+\frac{e_N}{n}|}}\frac{|x+\frac{e_N}{n}|^\gamma\tilde f_\gamma\left(|x+\frac{e_N}{n}|^2(1+\tau)^2\right)-|1+\tau|^{-\gamma}}{|\tau|^{1+2s}}\,d\tau.
\end{split}
\end{equation}
Putting together \eqref{3112eq2}-\eqref{2eq3}-\eqref{2eq4}-\eqref{2eq5} and using \cite[Proposition 3.7]{BGT} we infer that 
\begin{equation}\label{2eq6}
\begin{split}
\I^+_kv_\gamma(x)-\I^+_kv_\gamma\left(x+\frac{e_N}{n}\right)
&=c_k(\gamma)\left(|x|^{-\gamma-2s}-|x+\frac{e_N}{n}|^{-\gamma-2s}\right)\\
&\quad+
C_s\left(f_\gamma(|x|)-f_\gamma(|x+\frac{e_N}{n}|)\right)
\end{split}
\end{equation}
where $c_k(\gamma)$ is defined in \eqref{811eq1}-\eqref{811eq2} and $f_\gamma:(1,+\infty)\mapsto\R$ is the function defined by the formula
\begin{equation}\label{11mageq1}
f_\gamma(r)=r^{-\gamma-2s}\int\limits_{-1-\frac{1}{r}}^{-1+\frac{1}{r}}\frac{r^\gamma\tilde f_\gamma\left(r^2(1+\tau)^2\right)-|1+\tau|^{-\gamma}}{|\tau|^{1+2s}}\,d\tau.
\end{equation}
Consider now an orthonormal basis $\xi_1,\ldots,\xi_N$ such that $\xi_1=\hat{x}$.
Using the fact that 
$$
\I^+_kv_\gamma(x)=\I_{\hat{x}}v_\gamma(x)+(k-1)\I_{x^\perp}v_\gamma(x)=\sum_{i=1}^k\I_{\xi_i}v_\gamma(x)
$$and that
\begin{equation*}
\I^+_kv_\gamma\left(x+\frac{e_N}{n}\right)\geq\sum_{i=1}^k\I_{\xi_i}v_\gamma\left(x+\frac{e_N}{n}\right)
\end{equation*}
then, after multiplying both of sides of \eqref{2eq6} by $n$, we obtain the inequality
\begin{equation}\label{10mageq1}
\begin{split}
n\sum_{i=1}^k\left[\I_{\xi_i}v_\gamma(x)-\I_{\xi_i}v_\gamma\left(x+\frac{e_N}{n}\right)\right]&\geq c_k(\gamma)n\left(|x|^{-\gamma-2s}-|x+\frac{e_N}{n}|^{-\gamma-2s}\right)\\&\quad+
C_sn\left(f_\gamma(|x|)-f_\gamma(|x+\frac{e_N}{n}|)\right).
\end{split}
\end{equation}
Sending $n\to+\infty$ and using Lemma \ref{lem1}, we infer that
\begin{equation}\label{10mageq3}
\sum_{i=1}^k\I_{\xi_i}\left(-D_{x_N}v_\gamma\right)(x)\geq(\gamma+2s)c_k(\gamma)\frac{x_N}{|x|^{\gamma+2s+2}}-C_sf_\gamma'(|x|)\frac{x_N}{|x|}.
\end{equation}
Now we  estimate from below the right-hand side of \eqref{10mageq3}.

In view of \eqref{11mageq1} and using the fact that $\tilde f_\gamma(1)=1$, we have
$$
f_\gamma'(|x|)=-(\gamma+2s)|x|^{-\gamma-2s-1}\int\limits_{-1-\frac{1}{|x|}}^{-1+\frac{1}{|x|}}\frac{\tilde h_\gamma(|x|,\tau)}{|\tau|^{1+2s}}\,d\tau
$$
where
\begin{equation}\label{11mageq2}
\begin{split}
\tilde h_\gamma(|x|,\tau)&=(1-\frac{\gamma}{\gamma+2s})|x|^\gamma\tilde f_\gamma\left(|x|^2(1+\tau)^2\right)\\&\quad-\frac{2}{\gamma+2s}|x|^{\gamma+2}(1+\tau)^2\tilde f_\gamma'\left(|x|^2(1+\tau)^2\right)-|1+\tau|^{-\gamma}.
\end{split}
\end{equation}
 Since  $\tilde f_\gamma$ and $-\tilde f_\gamma'$ are both nonnegative in the interval $[0,1]$, we infer form \eqref{11mageq2} that
$$
\tilde h_\gamma(|x|,\tau)\geq-|1+\tau|^{-\gamma} \qquad\forall\tau\in\left[-1-\frac{1}{|x|},-1+\frac{1}{|x|}\right].
$$  
Hence
\begin{equation*}
\begin{split}
- f'_\gamma(|x|)&\geq-(\gamma+2s)|x|^{-\gamma-2s-1}\int\limits_{-1-\frac{1}{|x|}}^{-1+\frac{1}{|x|}}\frac{|1+\tau|^{-\gamma}}{|\tau|^{1+2s}}\,d\tau\\
&\geq-(\gamma+2s)\frac{|x|^{-\gamma-2s-1}}{\left(1-\frac{1}{|x|}\right)^{1+2s}}\int\limits_{-1-\frac{1}{|x|}}^{-1+\frac{1}{|x|}}|1+\tau|^{-\gamma}\,d\tau\\
&=-\frac{2(\gamma+2s)}{1-\gamma}\frac{1}{\left(1-\frac{1}{|x|}\right)^{1+2s}}\frac{1}{|x|^{2+2s}}\,.
\end{split}
\end{equation*}
Using \eqref{10mageq3} we deduce that
\begin{equation}\label{10mageq4} 
\sum_{i=1}^k\I_{\xi_i}\left(-D_{x_N}v_\gamma\right)(x)\geq (\gamma+2s)c_k(\gamma)\frac{x_N}{|x|^{\gamma+2s+2}}-\frac{2C_s(\gamma+2s)}{1-\gamma}\frac{1}{\left(1-\frac{1}{|x|}\right)^{1+2s}}\frac{x_N}{|x|^{2s+3}}.
\end{equation}
Since the above estimate holds for any $\gamma\in(0,1)$, in particular for $\gamma=\bar\gamma$ we also have
\begin{equation}\label{0101eq2} 
\sum_{i=1}^k\I_{\xi_i}\left(-D_{x_N}v_{\bar\gamma}\right)(x)\geq -\frac{2C_s(\bar\gamma+2s)}{1-\bar\gamma}\frac{1}{\left(1-\frac{1}{|x|}\right)^{1+2s}}\frac{x_N}{|x|^{2s+3}},
\end{equation}
where we have used that $c_k(\bar\gamma)=0$.\\
Now we fix $\gamma\in(\bar\gamma,1)$  and let $\psi$ be the function defined in \eqref{0101eq1}. Using the linearity of the operator $\sum_{i=1}^k\I_{\xi_i}u(x)$ in $u$, we infer from \eqref{10mageq4}-\eqref{0101eq2}  that 
\begin{equation*}
\begin{split}
\sum_{i=1}^k\I_{\xi_i}\psi(x)&\geq \frac{(\gamma+2s)c_k(\gamma)}{\bar\gamma}\frac{x_N}{|x|^{\gamma+2s+2}}-\frac{2C_s(\gamma+2s)}{(1-\gamma)\bar\gamma}\frac{1}{\left(1-\frac{1}{|x|}\right)^{1+2s}}\frac{x_N}{|x|^{2s+3}}\\
&\quad-\frac{2C_s(\bar\gamma+2s)}{(1-\bar\gamma)\bar\gamma}\frac{1}{\left(1-\frac{1}{|x|}\right)^{1+2s}}\frac{x_N}{|x|^{2s+3}}
\\&=\frac{\gamma+2s}{\bar\gamma}\frac{x_N}{|x|^{\gamma+2s+2}}\left[c_k(\gamma)-\frac{C(\gamma,\bar\gamma,s)}{\left(1-\frac{1}{|x|}\right)^{1+2s}}\frac{1}{|x|^{1-\gamma}}\right]
\end{split}
\end{equation*}
where $C(\gamma,\bar\gamma,s)$ is a positive constant just depending on $\gamma,\bar\gamma,s$. Since $c_k(\gamma)>0$ then we can pick $R_0$ sufficiently large such that 
\begin{equation}\label{10mageq5}
\sum_{i=1}^k\I_{\xi_i}\psi(x)\geq\frac{c_k(\gamma)(\gamma+2s)}{2\bar\gamma}\frac{x_N}{|x|^{\gamma+2s+2}} \quad\forall x\in\RN_+\backslash \overline{B_{R_0}(0)}.
\end{equation}
This implies the desired result.
\end{proof}

Using the strict subsolution $\psi$  provided by Proposition \ref{propSubsol}, we obtain a lower bound for the decay rate of nonnegative supersolutions of the homogeneous equation in the halfspace $\RN_+$.

\begin{proposition}\label{decayestimate}
Let $u\in LSC\left(\overline{\RN_+}\right)\cap\mathcal S$ be a supersolution of
\begin{equation}\label{3eq1}
\I^+_ku(x)=0\qquad\text{in $\RN_+$}
\end{equation}
such that 
$$
\text{$u>0$\; in $\overline{\RN_+}\quad$ and $\quad u\geq0$\; in $\RN$.} 
$$
Then there exists a positive constant $c=c(u,k,s)$  such that 
\begin{equation}\label{3eq8}
u(x)\geq c\,\frac{x_N}{|x|^{\bar\gamma+2}} \qquad\forall x \in\RN_+\backslash\overline{B_1(0)}.
\end{equation}
\end{proposition}
\begin{proof}
Let $\psi=\psi(x)$ and $R_0>1$ be given by Proposition \ref{propSubsol}. Set
$$
c=\inf_{x\in \overline{B_{R_0}(0)}\cap\RN_+}\frac{u(x)}{\psi(x)}
$$
which  is obviously a positive quantity since $\psi\equiv0$ on $\partial\RN_+$.

For $\varepsilon>0$ consider the function $u_\varepsilon(x)=u(x)+\varepsilon$ which is in turn a supersolution of \eqref{3eq1}. Set $R_1=\left(\frac{c}{\varepsilon}(1+\frac{\gamma}{\bar\gamma})\right)^{\frac{1}{\bar\gamma+1}}$. For $\varepsilon$ small enough we have $R_1>R_0$. Using \eqref{3012eq4}, for any   $x\in\mathbb R^N_+$ such that $|x|\geq R_1$ we obtain
\begin{equation}\label{3eq2}
c\psi(x)\leq c\left(1+\frac{\gamma}{\bar\gamma}\right)\frac{1}{|x|^{\bar\gamma+1}}\leq c\left(1+\frac{\gamma}{\bar\gamma}\right)\frac{1}{{R_1}^{\bar\gamma+1}}=\varepsilon<u_\varepsilon(x).
\end{equation}
By the definition of $c$ we also have
 \begin{equation}\label{3eq3}
u_\varepsilon(x)> u(x)\geq c\psi(x)\qquad\forall x\in \overline{B_{R_0}(0)}\cap\RN_+.
\end{equation}
Moreover
\begin{equation}\label{3eq4}
u_\varepsilon(x)> u(x)\geq0\geq c\psi(x)\qquad\forall x\in\overline{\RN_-}.
\end{equation}
We claim that $$u_\varepsilon(x)\geq c\psi(x)\qquad \forall x\in B_{R_1}(0)\cap\RN_+.$$ If not, in view of \eqref{3eq2}-\eqref{3eq3}-\eqref{3eq4}, there exists $x_\varepsilon\in\left\{x\in\RN_+:\,R_0<|x|<R_1\right\}$ such that 
$$
u_\varepsilon(x_\varepsilon)-c\psi(x_\varepsilon)=\min_{x\in\overline{\left\{x\in\RN_+:\,R_0<|x|<R_1\right\}}}(u_\varepsilon-c\psi)=\min_{x\in\RN}(u_\varepsilon-c\psi)<0.
$$
Hence testing $u_\varepsilon$ at $x_\varepsilon$ by using the function $c\psi(x)$, we infer that $$\I^+_k\psi(x_\varepsilon)\leq0,$$ so contradicting the strict inequality \eqref{3012eq1}.

\smallskip

Thus, $u(x)+\varepsilon\geq c\psi(x)$ for any $x\in \RN_+$. Letting $\varepsilon\to0^+$ we arrive at $$u(x)\geq c\psi(x)\qquad\forall x\in\RN_+,$$ which conclude the proof since, by \eqref{3012eq2}, $\psi(x)\geq\frac{x_N}{|x|^{\bar\gamma+2}}$ for any $x\in\RN_+$ such that $|x|>1$.
\end{proof}

\begin{proof}[Proof of Theorem \ref{th3} (case $p>1$)] 
By contradiction let us suppose that $u$ is a nontrivial solution of \eqref{eq21}. In view of the strong maximum principle, see e.g. \cite[Proposition 4.7 (i)]{BGS}, $u$ is in fact strictly positive in $\RN_+$. Moreover we can further assume that $u(x)>0$ for any $x\in\overline{\RN_+}$,  replacing if necessary $u(x)$ by $u_1(x)=u(x+e_N)$ which is in turn a solution \eqref{eq21}.

\smallskip
\noindent
In view of Proposition \ref{decayestimate} we infer that 
$$
M_1(R)=\min_{\overline{B_R(5Re_N)}} u\geq \frac{c}{R^{\bar\gamma+1}}\qquad\forall R\geq1,
$$
$c=c(u,k,s)$ being a positive constant.  Combining this estimate with the inequality  \eqref{0111peq3}, we obtain
\begin{equation*}
1\leq C R^{(\bar\gamma+1)(p-1)-2s} \qquad\forall R\geq1,
\end{equation*}
for some $C>0$. This leads to a  contradiction for $R\to+\infty$ since $(\bar\gamma+1)(p-1)-2s<0$.
\end{proof}

\bigskip

Now we are concerned with Theorem \ref{th4}. The main difference with respect to Theorem \ref{th3} is the lack of a fundamental solution, bounded at infinity, for the operator $\I^+_1$ whenever $s\in\left[\frac12,1\right)$. In fact, the function $w_\gamma(x)=|x|^{-\gamma}$ is now, for any $\gamma\in(0,1)$, only a strict subsolution of the homogeneous equation $\I^+_1u=0$ in $\RN\backslash\left\{0\right\}$. This subsolution property is however sufficient to prove that $u(x)=0$ is the unique solution of \eqref{eq21} provided $p<1+2s$. Nevertheless, this range of $p$ can be improved, that is $p<\frac{1}{1-s}$, using the fact that if $s>\frac12$, then the function $w(x)=-|x|^{2s-1}$ is solution (unbounded at infinity) for $\I^+_1$ in $\RN\backslash\left\{0\right\}$. 

\smallskip

We shall treat  the cases $s=\frac12$ and $s>\frac12$ separately and show how to adapt the arguments used in the proof of Theorem \ref{th3} to this setting.

\subsubsection{The case  $\I^+_1$ with $s=\frac12$}

\begin{proposition}\label{propSubsol2}
Let $\gamma\in(0,1)$. Then there exist a constant $R_0>1$ and a function $\psi\in C^2(\RN)\cap L^\infty(\RN)$, depending on $\gamma$  and $s$, such that 
\begin{equation*}
\I^+_1\psi(x)>0 \qquad \forall x\in\RN_+\backslash \overline{B_{R_0}(0)}
\end{equation*}
\begin{equation*}
\psi(x)=\frac{x_N}{|x|^{\gamma+2}} \;\quad\forall x\in\RN_+\backslash \overline{B_1(0)}
\end{equation*}
\begin{equation*}
\psi(x)>0\quad\text{if\; $x_N>0$},\;\;\;\psi(x)=0\quad\text{if\; $x_N=0$},\;\;\;\psi(x)<0 \quad\text{if\; $x_N<0$}.
\end{equation*}
\end{proposition}
\begin{proof}
Set 
$$
\psi(x)=-\frac1\gamma D_{x_N}v_\gamma(x),
$$
being $v_\gamma\in C^3(\RN)\cap L^\infty(\RN)$ the radial function defined by \eqref{3eq11}-\eqref{3012eq5}. The inequality \eqref{10mageq4}, in the case $k=1$, reads as
$$
\I_{\hat x}\psi(x)\geq \frac{(\gamma+2s)\hat c(\gamma)}{\gamma}\frac{x_N}{|x|^{\gamma+2s+2}}-\frac{2C_s(\gamma+2s)}{(1-\gamma)\gamma}\frac{1}{\left(1-\frac{1}{|x|}\right)^{1+2s}}\frac{x_N}{|x|^{2s+3}}
 $$
with $\hat c(\gamma)$ defined in \eqref{811eq2}. Hence, we have
$$
\I^+_1\psi(x)\geq\frac{(\gamma+2s)}{\gamma}\frac{x_N}{|x|^{\gamma+2s+2}}\left[\hat c(\gamma)-\frac{2C_s}{(1-\gamma){\left(1-\frac{1}{|x|}\right)}^{1+2s}}\frac{1}{|x|^{1-\gamma}}\right].
$$
Since $\hat c(\gamma)>0$,  by choosing $R_0=R_0(\gamma,s)$ sufficiently large,  we conclude that 
\begin{equation*}
\I^+_1\psi(x)\geq\frac{\hat c(\gamma)(\gamma+2s)}{2\gamma}\frac{x_N}{|x|^{\gamma+2s+2}} \quad\forall x\in\RN_+\backslash \overline{B_{R_0}(0)}.
\end{equation*}
\end{proof} 

Using the subsolution $\psi$ provided by Proposition \ref{propSubsol2} and arguing as in the proof of Proposition \ref{decayestimate}, we obtain the following

\begin{proposition}\label{decayestimate2}
Let $u\in LSC\left(\overline{\RN_+}\right)\cap\mathcal S$ be a supersolution of
\begin{equation*}
\I^+_1u(x)=0\qquad\text{in $\RN_+$}
\end{equation*}
such that 
$$
\text{$u>0$\; in $\overline{\RN_+}\quad$ and $\quad u\geq0$\; in $\RN$.}
$$
Then for any $\gamma\in(0,1)$, there exists a positive constants $c=c(u,\gamma,s)$  such that 
\begin{equation*}
u(x)\geq c\,\frac{x_N}{|x|^{\gamma+2}} \qquad\forall x \in\RN_+\backslash\overline{B_1(0)}.
\end{equation*}
\end{proposition}

\begin{proof}[Proof of Theorem \ref{th4}  (case $s=\frac12$)] Thanks to Proposition \ref{1novprop1}, we only need to treat the case $p>1$. 
Suppose by contradiction that $u$ is a nontrivial solution of \eqref{eq21}. As in the proof of Theorem \ref{th3} we may further assume that $u$ is positive in $\overline\RN_+$ and, using \eqref{0111peq3} and Proposition \ref{decayestimate2}, that for any $\gamma\in(0,1)$ 
\begin{equation}\label{911eq1}
1\leq C R^{(\gamma+1)(p-1)-1}\qquad\forall R\geq1
\end{equation}
with $C$ positive constant depending also on $\gamma$. Now, since $1<p<\frac{1}{1-s}=2$, we can pick $\gamma\in(0,1)$ such that $(\gamma+1)(p-1)-1<0$, leading to a contradiction in \eqref{911eq1} for sufficiently large $R$.
\end{proof}

\subsubsection{The case $\I^+_1$ with $\frac12<s<1$}

Let  $\bar\gamma=2s-1$ and set for $\gamma\in(0,\bar\gamma]$
$$
w_{-\gamma}(x)=-|x|^{\gamma}\,,\qquad x\neq0.
$$
By \cite[Theorem 3.4 and Lemma 3.6]{BGT}, for any $x\neq0$, it holds that
$$
\I^+_1w_{-\gamma}(x)=\I_{\hat x}w_{-\gamma}(x)=\hat c(\gamma)|x|^{\gamma-2s}
$$
where 
$$
\hat c(\gamma)=C_s\text{P.V.}\,\int\limits_{-\infty}^{+\infty}\frac{1-|1+\tau|^{\gamma}}{|\tau|^{1+2s}}\,d\tau
$$
and $\hat c(\gamma)>0$ whenever $\gamma<\bar\gamma$ and $\hat c(\bar\gamma)=0$.

In a similar way as \eqref{3eq11}-\eqref{3012eq5}, we set for $x\in\RN$
\begin{equation}\label{0601eq3}
v_{-\gamma}(x)=\tilde g(|x|^2),
\end{equation}
where $\tilde g\in C^3([0,+\infty))$ is such that 
\begin{equation}\label{0601eq4}
\begin{cases}
\tilde g(r)=-r^{\frac\gamma2} & \text{if $r\geq1$}\\
\tilde g(r)\leq-r^{\frac\gamma2} & \text{if $r\in[0,1]$}\\
\text{$\tilde g(r)$ and $\tilde g''(r)$} & \text{are convex in $[0,+\infty)$.} 
\end{cases}
\end{equation}
Notice that $v_{-\gamma}\in C^3(\RN)$ and all the derivatives, up to order 3, of $v_{-\gamma}$ are bounded in $\RN$ but, differently from \eqref{3eq11}-\eqref{3012eq5}, the function $v_{-\gamma}$ is not bounded in $\RN$.

\begin{proposition}\label{propSubsol3}
There exist a constant $R_0>1$ and a function $\psi\in C^2(\RN)\cap L^\infty(\RN)$, depending on  $s$, such that 
\begin{equation}\label{0601eq5}
\I^+_1\psi(x)>0 \qquad \forall x\in\RN_+\backslash \overline{B_{R_0}(0)}
\end{equation}
\begin{equation}\label{0601eq6}
\psi(x)\geq\frac{x_N}{|x|^{3-2s}} \;\quad\forall x\in\RN_+\backslash \overline{B_1(0)}\,\quad{and}\,\quad\lim_{\stackrel{|x|\to+\infty}{x_N\neq0}
}\frac{\psi(x)}{\frac{x_N}{|x|^{3-2s}}}=1
\end{equation}
\begin{equation}\label{0601eq7}
\psi(x)>0\quad\text{if\; $x_N>0$},\;\;\;\psi(x)=0\quad\text{if\; $x_N=0$},\;\;\;\psi(x)<0 \quad\text{if\; $x_N<0$}.
\end{equation}
\end{proposition}
\begin{proof}[Sketch of the proof.] Fix $\gamma<\bar\gamma=2s-1$ and set
$$
\psi(x)=-\frac{1}{\bar\gamma}\left(D_{x_N}v_{-\bar\gamma}(x)+D_{x_N}v_{-\gamma}(x)\right),
$$
$v_{-\gamma}(x)$ be the function defined in \eqref{0601eq3}-\eqref{0601eq4}. \\
Conditions \eqref{0601eq6}-\eqref{0601eq7} easily follows from the definition of $\psi$. As far as \eqref{0601eq5} is concerned, using Lemma \ref{lem1} and arguing as in the proof of Proposition \ref{propSubsol} with small modifications, we obtain that 
$$
\I^+_1\psi(x)\geq\I_{\hat x}\psi(x)\geq\frac{\hat c(\gamma)(2s-\gamma)}{2\bar\gamma}\frac{x_N}{|x|^{2s+2-\gamma}}
\qquad\forall \RN_+\backslash\overline{B_{R_0}(0)}$$
provided $R_0$ is sufficiently large. Hence the conclusion follows.
\end{proof}
As a consequence we obtain the following  
\begin{proposition}\label{decayestimate3} 
Let $u\in LSC\left(\overline{\RN_+}\right)\cap\mathcal S$ be a supersolution of
\begin{equation*}
\I^+_1u(x)=0\qquad\text{in $\RN_+$}
\end{equation*}
such that 
$$
\text{$u>0$\; in $\overline{\RN_+}\quad$ and $\quad u\geq0$\; in $\RN$.}
$$
Then there exists a positive constants $c=c(u,s)$  such that 
\begin{equation}\label{0701eq1}
u(x)\geq c\,\frac{x_N}{|x|^{3-2s}} \qquad\forall x \in\RN_+\backslash\overline{B_1(0)}.
\end{equation}
\end{proposition}
Using the estimate \eqref{0701eq1} and reasoning in a similar manner as done in the proofs of Theorem \ref{th3} and Theorem \ref{th4} (case $s=\frac12$), we obtain the desired result in the range $s\in(\frac12,1)$. Details are left to the the reader.

\bigskip

We end this section by pointing out that, under the further assumption $u \in LSC(\R^N)$,  the thesis of  Theorems \ref{th3}-\ref{th4} is $u\equiv0$ in $\R^N$, the conclusion being a consequence of \cite[Proposition 4.7 (i)]{BGS}. Interestingly the situation is different for $\I^-_N$. As showed in the next example, for any $p>0$, there exist solutions $u\in LSC(\R^N)\cap L^\infty(\R^N)$  of 
\begin{equation}\label{311224eq1}
\I^-_Nu(x)+u^p(x)\leq0 \quad\text{in}\;\;\R^N_+
\end{equation}
satisfying the following properties:
\begin{equation}\label{311224eq2}
u\geq0\quad\text{in}\;\;\RN\;,\qquad u\equiv0\quad\text{in}\;\;\RN_+\qquad\text{and}\qquad u\not\equiv0\quad\text{in}\;\;\RN\backslash\R^N_+.
\end{equation}
\begin{example}
{\rm
Fix any $r>0$ and any $y\in\R^N\backslash\overline{\R^N_+}$ such that 
\begin{equation}\label{311224eq3}
y_N\leq-\sqrt{2}r.
\end{equation}
Let $f\in LSC(B_r(y))\cap L^\infty(B_r(y))$ be such that $f\geq0$ and $f\not\equiv0$ in $B_r(y)$. Consider the function
$$
u(x)=\begin{cases}
f(x) & \text{if $x\in B_r(y)$}\\
0 & \text{otherwise}.
\end{cases}
$$
By construction $u\in LSC(\R^N)\cap L^\infty(\R^N)$ satisfies the conditions \eqref{311224eq2}.  We now prove that 
$$
\I^-_Nu(x)\leq0 \quad\text{in}\;\;\R^N_+.
$$
For this, let $x\in\R^N_+$ and let $\xi_1,\ldots,\xi_N$ be the orthonormal basis such that 
\begin{equation}\label{311224eq4}
\left\langle \widehat{x-y},\xi_i\right\rangle=\frac{1}{\sqrt{N}}\qquad i=1,\ldots,N.
\end{equation}
We claim that 
\begin{equation}\label{311224eq5}
x+\tau\xi_i\notin \overline{B_r(y)}\qquad\forall \tau\in\R,\;i=1,\ldots,N.
\end{equation}
Once this is proved, then we easily infer that $$\I_{\xi_i}u(x)=0\qquad\text{ for any $i=1,\ldots,N$.}$$ Hence we conclude
$$
\I^-_Nu(x)\leq\sum_{i=1}^N\I_{\xi_i}u(x)=0.$$ In order to prove \eqref{311224eq5},  using  \eqref{311224eq4}  we have
\begin{equation}\label{311224eq6}
\begin{split}
|x+\tau\xi_i-y|^2&=|x-y|^2+\frac{2}{\sqrt{N}}|x-y|\tau+\tau^2\\
&\geq|x-y|^2-\frac1N|x-y|^2\geq\frac{|x-y|^2}{2}.
\end{split}
\end{equation}
By the assumption \eqref{311224eq3} and $x\in\R^N_+$, we also have
$$
|x-y|\geq x_N-y_N>-y_N>\sqrt{2}r.$$
Thus, from \eqref{311224eq6}, we conclude  
$$
|x+\tau\xi_i-y|^2>r^2\qquad\forall \tau\in\R,\;i=1,\ldots,N.
$$
}
\end{example}


\section{Existence of supersolutions}\label{Ex}
\subsection{Theorem \ref{th1}}

\begin{proof}[Proof of Theorem \ref{th1}]
For any nonnegative integer $n$ and any $\varepsilon\in(0,\frac12)$, let 
$$
u_{\varepsilon,n}(t)=\left(\varepsilon^2-(t-n-\varepsilon)^2\right)^s_+.
$$
Such function has constant one dimensional fractional laplacian in the interval $(n,n+2\varepsilon)$, in fact
\begin{equation}\label{eq3}
(-\Delta)^s u_{\varepsilon,n}(t):=C_s\,{\rm P.V.}\int\limits_{-\infty}^{+\infty}\frac{u_{\varepsilon,n}(t)-u_{\varepsilon,n}(\tau)}{|t-\tau|^{1+2s}}\,d\tau=C_s\beta(1-s,s)\qquad\forall t\in(n,n+2\varepsilon)
\end{equation}
where $\beta(\cdot,\cdot)$ is the special Beta function. For a proof of \eqref{eq3} see \cite[Section 2.6]{BV}
and use the rescaling property and the translation invariance of  $(-\Delta)^s$.\\
Note that $${\rm supp}\,u_{\varepsilon,n}=[n,n+2\varepsilon]\qquad\text{and}\qquad{\rm supp}\,u_{\varepsilon,n}\cap {\rm supp}\,u_{\varepsilon,m}=\emptyset$$
for $n\neq m$, since $\varepsilon\in(0,\frac12)$.

Set
\begin{equation*}
u_\varepsilon(x):=\sum_{n=0}^{+\infty}u_{\varepsilon,n}(x_N)\,,\qquad x\in\RN.
\end{equation*}
By construction $u_\varepsilon\in C^{0,s}(\RN)\cap L^\infty(\RN)$, $u\geq0$ in $\RN$ and $u(x)=0$ for any  $x\in\RN\backslash\mathbb R^N_+$. Moreover
$$
\limsup_{x_N\to+\infty} u_\varepsilon(x',x_N)=\varepsilon^{2s}\qquad\text{uniformly w.r.t. $x'\in\mathbb R^{N-1}$}.
$$
We claim that for $\varepsilon=\varepsilon(s,p)$ small enough, then $u_\varepsilon$ satisfies in the viscosity sense the inequality
\begin{equation}\label{eq4}
\I^-_ku_\varepsilon(x)+u_\varepsilon^p(x)\leq0\quad\text{in $\RN_+$}.
\end{equation}
Let $x\in\RN_+$. We distinguish three different cases.

\medskip

\noindent
\textbf{Case 1:} $x_N\in(m,m+2\varepsilon)$ for some nonnegative integer $m$.

\smallskip
\noindent
In this case $u_\varepsilon$ is twice differentiable in a neighborhood of $x$, hence the operator $\I^-_ku_\varepsilon$ can be evaluated classically at $x$. Moreover, denoting by $\left\{e_1,\ldots,e_N\right\}$ the canonical basis of $\RN$, it turns out that
\begin{equation}\label{eq5}
\I^-_ku_\varepsilon(x)\leq\sum_{i=N-k+1}^N\I_{e_i}u_\varepsilon(x)=\I_{e_N}u_\varepsilon(x)\,.
\end{equation}
In \eqref{eq5} we have used the minimality of $\I^-_k$, within the class of $k$-orthonormal sets in $\RN$, for the first inequality and the fact that $u_\varepsilon(x+\tau e_i)=u_\varepsilon(x)$ for any $\tau\in\mathbb R$ and $i<N$, for the last equality.\\
As far as $\I_{e_N}u_\varepsilon(x)$ is concerned, using \eqref{eq3}, we have
\begin{equation}\label{eq6}
\begin{split}
\I_{e_N}u_\varepsilon(x)&=C_s\,{\rm P.V.}\int\limits_{-\infty}^{+\infty}\frac{u_{\varepsilon}(x+\tau e_N)-u_{\varepsilon}(x)}{|\tau|^{1+2s}}\,d\tau\\
&=C_s\,{\rm P.V.}\int\limits_{-\infty}^{+\infty}\frac{u_{\varepsilon,m}(x_N+\tau)-u_{\varepsilon,m}(x_N)}{|\tau|^{1+2s}}\,d\tau+C_s\,\int\limits_{-\infty}^{+\infty}\frac{\sum_{n\neq m}u_{\varepsilon,n}(x_N+\tau)}{|\tau|^{1+2s}}\,d\tau\\
&= -C_s\beta(1-s,s)+C_s\int\limits_{\left(-(1-2\varepsilon),1-2\varepsilon\right)^c}\frac{\sum_{n\neq m}u_{\varepsilon,n}(x_N+\tau)}{|\tau|^{1+2s}}\,d\tau
\end{split}
\end{equation}
where in the last equality we have used the fact that $$\sum_{n\neq m}u_{\varepsilon,n}(x_N+\tau)=0\qquad\forall\tau\in\left(-(1-2\varepsilon),1-2\varepsilon\right).$$ 
Moreover, since
\begin{equation*}
0\leq\sum_{n\neq m}u_{\varepsilon,n}(x_N+\tau)\leq\sum_{n=0}^{+\infty}u_{\varepsilon,n}(x_N+\tau)\leq \varepsilon^{2s}\qquad\forall\tau\in\mathbb R,
\end{equation*}
we easily obtain from \eqref{eq6} that
\begin{equation}\label{eq8}
\I_{e_N}u_\varepsilon(x)\leq -C_s\beta(1-s,s)+C_s\frac{(1-2\varepsilon)^{-2s}}{s}\varepsilon^{2s}.
\end{equation}
Now we use  \eqref{eq5} and \eqref{eq8} to infer that
$$
\I^-_ku_\varepsilon(x)+u_\varepsilon^p(x)\leq -C_s\beta(1-s,s)+C_s\frac{(1-2\varepsilon)^{-2s}}{s}\varepsilon^{2s}+\varepsilon^{2sp}\leq0
$$
choosing $\varepsilon=\varepsilon(s,p)$ small enough.   

\medskip

\noindent
\textbf{Case 2:} $x_N\in(m+2\varepsilon,m+1)$ for some nonnegative integer $m$.

\smallskip
\noindent
In this case $u_\varepsilon=0$ in a neighborhood of $x$. In addition 
\begin{equation}\label{eq9}
u_\varepsilon(x+\tau e_i)=0\qquad\forall\tau\in\mathbb R,\,\forall i=1,\ldots,N-1.
\end{equation}
By \eqref{eq9} it follows that
$$
\I^-_ku_\varepsilon(x)+u^p_\varepsilon(x)\leq\sum_{i=1}^k\I_{e_i}u_\varepsilon(x)=0.
$$

\smallskip

\noindent
\textbf{Case 3:} $x_N=m$ or $x_N=m+2\varepsilon$ for some positive integer $m$.

\smallskip
\noindent
Differently from the previous cases, now $u_\varepsilon$ is not differentiable at $x$. We shall prove that the inequality
$$
\I^-_ku_\varepsilon(x)+u_\varepsilon^p(x)\leq0
$$
holds in the viscosity sense. For this let $\delta>0$ and $\varphi\in C^2(\overline{B_\delta(x)})$ such that 
$$
\varphi(x)=u_\varepsilon(x)=0\qquad\text{and}\qquad\varphi(y)\leq u_\varepsilon(y) \quad\forall y\in B_\delta(x)\subset\R^N_+.
$$
Set
$$
\psi(y)=\begin{cases}
\varphi(y) & \text{if $y\in B_\delta(x)$}\\
u_\varepsilon(y) & \text{otherwise}.
\end{cases}
$$
Since for any $i=1,\ldots,N-1$ it holds that 
\begin{equation*}
\begin{split}
\psi(x+\tau e_i)&\leq0 \qquad\forall \tau\in(-\delta,\delta)\\
\psi(x+\tau e_i)&=0 \qquad\forall \tau\in(-\delta,\delta)^c,
\end{split}
\end{equation*}
then we conclude 
\begin{equation*}
\I^-_k\psi(x)+u_\varepsilon^p(x)=\I^-_k\psi(x)\leq\sum_{i=1}^k\I_{e_i}\psi(x)=\sum_{i=1}^k C_s\,\int\limits_{0}^{\delta}\frac{\psi(x+\tau e_i)+\psi(x-\tau e_i)}{\tau^{1+2s}}\,d\tau\leq0.
\end{equation*}
\end{proof}

\subsection{Theorem \ref{thIN}}


We start by recalling that the constant $\tilde \gamma$ is the unique positive number vanishing the function $c(\gamma)$ defined, for all $\gamma>0$, by 
\begin{equation}\label{0801eq2}
c(\gamma)=\int\limits_0^{+\infty}\frac{\left(1+\tau^2+\frac{2}{\sqrt N}\tau\right)^{-\frac\gamma{2}}
+\left(1+\tau^2-\frac{2}{\sqrt N}\tau\right)^{-\frac{\gamma}2}-2}{\tau^{1+2s}}\,d\tau. 
\end{equation}
The key properties of $c(\gamma)$ are (see \cite[Lemma 4.8]{BGT}): 
\begin{equation}\label{0701peq1}
c(0^+)=0\,,\quad c'(0^+)<0\,,\quad \lim_{\gamma\to+\infty}c(\gamma)=+\infty \ \ \text{ and } \ \ 
c \ \text{ is convex in } (0,+\infty).
\end{equation}
We introduce the function $c_N^+(\gamma)$ defined, for any $\gamma>0$, by
\begin{equation*}
c_N^+(\gamma)=Nc(\gamma)-\int\limits_{\sqrt N}^{+\infty} \frac{\left(1+\tau^2-\frac 2{\sqrt N}\tau\right)^{-\frac\gamma 2}}{\tau^{1+2s}}\,d\tau.
\end{equation*}
The number $\gamma^+$ appearing in Theorem \ref{thIN} is given by the following 
\begin{lemma} \label{T49-1}There is a number $\gamma^+ \in(\tilde\gamma,\,+\infty)$ such that 
\[
c_N^+(\gamma)<0 \ \ \text{ for } \gamma\in(0,\, \gamma^+)\,,\ \ c_N^{+}(\gamma^+)=0 \ \ \text{ and } \ \ c_N^+(\gamma)>0 \ \ \text{ for } \gamma\in(\gamma^+,+\infty). 
\]
\end{lemma}
\begin{proof}
In view of \eqref{0701peq1} we infer that 
\begin{equation}\label{0701peq2}
c_N^{+}(0^+)<0 \ \ \text{ and } \ \ c_N^+(\gamma)<c(\gamma) \ \ \text{ for } \gamma\in(0,+\infty).
\end{equation}
Since, for $\tau>\sqrt N$,
\[
1+\tau^2-\frac 2{\sqrt{N}}\,\tau=1+\tau\,\frac{\sqrt N\tau-2}{\sqrt N}>1+\tau\,\frac{N-2}{\sqrt N}
\geq 1,
\]
it follows that 
\[
\lim_{\gamma\to+\infty} (c_N^+(\gamma)-Nc(\gamma))=0,
\]
and 
\begin{equation}\label{0701peq3}
\lim_{\gamma \to+\infty} c_N^+(\gamma)=+\infty.
\end{equation}
Moreover $c_N^+\in C^2(0,+\infty)$ and by a straightforward computation we see that
\begin{equation*}
\begin{split}
(c_N^+)''(\gamma)&=Nc''(\gamma)-\frac14\int\limits_{\sqrt N}^{+\infty} \frac{f(-\tau)}{\tau^{1+2s}}\,d\tau\\
&=\frac N4\int\limits_0^{\sqrt{N}}\frac{f(\tau)+f(-\tau)}{\tau^{1+2s}}\,d\tau+\frac N4\int\limits_{\sqrt{N}}^{+\infty}\frac{f(\tau)}{\tau^{1+2s}}\,d\tau+\left(\frac N4-\frac14\right)\int\limits_{\sqrt{N}}^{+\infty}\frac{f(-\tau)}{\tau^{1+2s}}\,d\tau
\end{split}
\end{equation*} 
where $$f(\tau)=\left(1+\tau^2+\frac{2}{\sqrt N}\tau\right)^{-\frac\gamma{2}}
\log^2\left(1+\tau^2+\frac{2}{\sqrt N}\tau\right).$$
Since $f(\tau)\geq0$ for any $\tau\in\R$ and $f(\tau)>0$ for any $\tau\in\R\backslash\left\{-\frac{2}{\sqrt{N}},0\right\}$, then $c_N^+$ is strictly convex in $(0,+\infty)$. From this and using \eqref{0701peq1}, \eqref{0701peq2} and \eqref{0701peq3} we obtain the result.
\end{proof}

Let $v_\gamma\in C^3(\RN)\cap L^\infty(\RN)$ be the function defined by \eqref{3eq11}-\eqref{3012eq5} and consider
\begin{equation}\label{49-1}
u_\gamma(x)=\begin{cases}
0 & \text{for $x\in\RN\backslash\RN_+$}\\
v_\gamma(x) & \text{for $x\in\RN_+$} 
\end{cases}=
\begin{cases}
0 & \text{for $x\in\RN\backslash\RN_+$}\\
\tilde f(|x|^2) & \text{for $x\in\RN_+\,,\;|x|\leq1$}\\
|x|^{-\gamma} & \text{for $x\in\RN_+\,,\;|x|>1$}. 
\end{cases}
\end{equation}
Notice that, for $x\in\RN_+$, $\I_\xi u_\gamma(x)$ is well defined for any direction $\xi\in\mathbb S^{N-1}$ and for any positive $\gamma$ (not only $\gamma\in(0,1)$).

\begin{proposition} \label{T49-2}Let $R={\sqrt{\frac{N}{N-1}}}$. Then the function $u_\gamma$ given by \eqref{49-1} satisfies 
\[
\I_N^-u_\gamma(x)\leq C_s c_N^+(\gamma) u^{1+\frac{2s}{\gamma}}(x) \ \ \text{ in }\, \RN_+\backslash B_R(0).
\]
\end{proposition}
\begin{proof}
For $x\in\RN_+$ we choose an orthonormal basis $\left\{\xi_i:\,i\in I\right\}$ where $I=\left\{1,\ldots,N\right\}$ of $\RN$ so that 
$$
\hat x=\frac 1{\sqrt N}(\xi_1+\cdots+\xi_N).
$$
It follows that $\left\langle \hat x, \xi_i\right\rangle=\frac{1}{\sqrt N}$ and
\[
\left\langle\hat x, e_N\right\rangle=\frac 1{\sqrt N} \left(\left\langle\xi_1, e_N\right\rangle+\cdots+\left\langle\xi_N, e_N\right\rangle\right).
\]
The later implies that there exists $\bar i\in I$ such that 
\[
\left\langle\hat x, e_N\right\rangle\leq N\frac 1{\sqrt N}\left\langle\xi_{\bar i}, e_N\right\rangle,  
\]
i.e., 
\[
\left\langle x,e_N\right\rangle \leq \sqrt N|x| \left\langle\xi_{\bar i},e_N\right\rangle. 
\]
This implies that for some $\tau_{\bar i}\in(0,\,|x|\sqrt N]$ it holds that 
\begin{equation}\label{49-2}
x+\tau \xi_{\bar i} \in\begin{cases} 
\R_+^N & \text{ for } \tau>-\tau_{\bar i} \\ 
\R^N\backslash \R_+^N & \text{ for } \tau\leq -\tau_{\bar i}. 
\end{cases}
\end{equation}
We define the subsets $I_+,\,I_-, \, I_0$ of $I$ such that 
\[
I=I_+\cup I_-\cup I_0,\quad \text{ and } \quad
\left\langle \xi_i,e_N\right\rangle
\begin{cases} 
>0 & \text{ for } i\in I_+\\
<0 & \text{ for } i\in I_-\\
=0 & \text{ for } i\in I_0.
\end{cases}
\]
Note $\bar i\in I_+$, that if $i\in I_+$, then there exists $\tau_i>0$ such that 
\[
x+\tau\xi_i \in\begin{cases} \R_+^N &\text{ for } \tau>-\tau_i \\ 
\R^N\backslash\R_+^N & \text{ for }\tau \leq -\tau_i ,
\end{cases}
\] 
that if $i\in I_-$, then there exists $\tau_i>0$ such that 
\[
x+\tau \xi_i \in\begin{cases}
\R^N\backslash\R_+^N & \text{ for } \tau\geq \tau_i\\
\R_+^N &\text{ for } \tau <\tau_i
\end{cases}
\] 
and that if $i\in I_0$, then $x+\tau\xi_i\in\R_+^N$ for $\tau\in\R$.
 
\smallskip

Suppose also that $|x|\geq\sqrt{\frac{N}{N-1}}$. A straightforward computation yields
$$
|x+\tau\xi_i|\geq1\qquad\forall  \tau\in\R,\;i=1,\ldots,N.
$$ 
Hence $v_\gamma(x\pm\tau\xi_i)={|x\pm\tau\xi_i|}^{-\gamma}$ and 
\begin{equation}\label{0801eq1}
\I_{\xi_i}v_\gamma(x)=C_sc(\gamma)|x|^{-\gamma-2s}
\end{equation}
with $c(\gamma)$ defined by \eqref{0801eq2}.\\
Consider the case when $i\in I_+$. Using \eqref{0801eq1} we have
\begin{align*}
\I_{\xi_i}u_\gamma(x)&=C_s\int\limits_0^{\tau_i} \frac{u_\gamma(x+\tau\xi_i)+u_\gamma(x-\tau \xi_i)-2u_\gamma(x)}{\tau^{1+2s}}\,d\tau
 +C_s\int\limits_{\tau_i}^{+\infty} \frac{u_\gamma(x+\tau\xi_i)-2u_\gamma(x)}{\tau^{1+2s}}\,d\tau
\\& =C_s\int\limits_0^{\tau_i} \frac{v_\gamma(x+\tau\xi_i)+v_\gamma(x-\tau \xi_i)-2 v_\gamma(x)}{\tau^{1+2s}}\,d\tau
\\&\quad+C_s\int\limits_{\tau_i}^\infty \frac{v_\gamma(x+\tau\xi_i)+v_\gamma(x-\tau\xi_i)-2v_\gamma(x)}{\tau^{1+2s}}\,d\tau-C_s\int\limits_{\tau_i}^{+\infty} \frac{v_\gamma(x-\tau\xi_i)}{\tau^{1+2s}}\,d\tau
\\&=C_sc(\gamma)|x|^{-\gamma-2s}  -C_s\int\limits_{\tau_i}^{+\infty} \frac{v_\gamma(x-\tau\xi_i)}{\tau^{1+2s}}\,d\tau.
\end{align*}
Since $v_\gamma\geq 0$, we may conclude that 
\begin{equation}\label{49-3}
\I_{\xi_i}u_\gamma(x)
\leq C_sc(\gamma)|x|^{-\gamma-2s}\qquad i\in I_+. 
\end{equation}
When $i=\bar i$, we deal more carefully with the term  
\[
\int\limits_{\tau_{\bar i}}^{+\infty} \frac{v_\gamma(x-\tau\xi_{\bar i})}{\tau^{1+2s}}\,d\tau.
\]
In view of \eqref{49-2}, we have 
$\tau_{\bar i}\leq|x|\sqrt N$ and hence 
\[
\begin{split}
\int\limits_{\tau_{\bar i}}^{+\infty} \frac{v_\gamma(x-\tau\xi_{\bar i})}{\tau^{1+2s}}\,d\tau
&\geq \int\limits_{|x|\sqrt N}^{+\infty} \frac{v_\gamma(x-\tau\xi_{\bar i})}{\tau^{1+2s}}\,d\tau\\
& =|x|^{-\gamma-2s} \int\limits_{\sqrt N}^{+\infty} \frac{v_\gamma(\hat x-\tau \xi_{\bar i})}{\tau^{1+2s}}\,d\tau
\\& =|x|^{-\gamma-2s} \int\limits_{\sqrt N}^{+\infty}
\frac {\left(1+\tau^2 -\frac 2{\sqrt N}\,\tau \right)^{-\frac \gamma 2}}{\tau^{1+2s}}\,d\tau.
\end{split}
\]
Hence, we have
\begin{equation} \label{49-4}
\I_{\xi_{\bar i}}u_\gamma(x)\leq C_s
 \left(c(\gamma)-\int\limits_{\sqrt N}^{+\infty} \frac{\left(1+\tau^2-\frac 2{\sqrt N} \tau\right)^{-\frac\gamma 2}}{\tau^{1+2s}}d\tau\right)|x|^{-\gamma-2s}. 
\end{equation}
Next, consider the case when $i\in I_-$. Similarly to \eqref{49-3}, we have
\begin{equation}\label{49-5}
\I_{\xi_i}u_\gamma(x)
\leq C_sc(\gamma)|x|^{-\gamma-2s}\qquad i\in I_-. 
\end{equation}
In the case when $i\in I_0$, we have $$u_\gamma(x\pm\tau\xi_i)=v_\gamma(x\pm\tau\xi_i)=|x\pm\tau\xi_i|^{-\gamma}\qquad\forall \tau\geq0.$$
Hence, by \eqref{0801eq1}, we obtain that 
\begin{equation} \label{49-6}
\I_{\xi_i}u_\gamma(x)=C_sc(\gamma)|x|^{-\gamma-2s}\qquad i\in I_0.
\end{equation}
Combining \eqref{49-3}--\eqref{49-6}, we conclude
\[
\I_{N}^- u_\gamma(x)\leq\sum_{i\in I}\I_{\xi_i} u_\gamma(x)\leq C_s c_N^+(\gamma) |x|^{-\gamma-2s} 
=C_sc_N^+(\gamma) u^{1+\frac {2s}\gamma}(x). 
\]
\end{proof}
\begin{remark}\label{T49-3}
{\rm Let $p>1+\frac {2s}{\gamma^+}$. Pick $\gamma\in(0,\,\gamma^+)$ 
such that $p=1+\frac{2s}{\gamma}$ and let $u_\gamma$ be the function provided by Proposition \ref{T49-2}. 
For $\varepsilon>0$  to be fixed,  compute that for $x\in\R_+^N\backslash B_R(0)$,
\[
\I_N^- (\varepsilon u_\gamma)(x)+(\varepsilon u_\gamma)^p(x) 
\leq \left(\varepsilon^p+\varepsilon C_sc_N^+(\gamma)\right)u_\gamma^p(x)
=\varepsilon \left(\varepsilon^{p-1}+C_sc_N^+(\gamma)\right)u_\gamma^p(x).
\]
Since $c_N^+(\gamma)<0$ by Lemma \ref{T49-1}, we may choose $\varepsilon>0$ small enough in such a way 
\[
\varepsilon^{p-1}+C_sc_N^+(\gamma)\leq 0. 
\]
Then the function $\varepsilon u_\gamma$ is a supersolution of 
\[
\I_N^- u+u^p =0 \ \ \text{ in } \R_+^N\backslash B_R(0).
\]
By translating downward the function $u_\gamma$, we obtain that the function $\phi(x)=\varepsilon u_\gamma\left(x+Re_N\right)$ is in turn a supersolution of
\[
\I_N^- u+u^p =0 \ \ \text{ in } \R_+^N.
\]
Notice that $\phi \geq 0$ in $\R^N$, $\phi\in LSC(\R^N)$ and that  
$\phi$ discontinuous in $\R^N$. Note furthermore that 
$\phi(x)=0$ for $x\in\RN\backslash(\RN_+-Re_N)$, $\phi(x)>0$ for $x\in\RN_+-Re_N$ 
and that $\phi\in C(\RN_+-Re_N)$.}
\end{remark}

\begin{proof}[Proof of Theorem \ref{thIN}] 
Let $p>1+\frac{2s}{\gamma^+}$, where $\gamma^+>\tilde\gamma$ is defined in Lemma \ref{T49-1}.  Choose $\gamma\in(0,\,\gamma^+)$ such that $p=1+\frac{2s}{\gamma}$. In view of Proposition \ref{T49-2}, the function $$\phi(x)=u_\gamma(x+Re_N)\,,\qquad R=\sqrt{\frac{N}{N-1}}$$ satisfies
\[
\I_N^- \phi(x)+\alpha \phi^p(x)\leq 0 \ \ \text{ in } \R_+^N, 
\]
where $\alpha=-c_N^+(\gamma)C_s$ is positive because of $\gamma<\gamma^+$.

\smallskip

Fix any $\mu\in(0,\,s)$ and set 
\[
z(x)=(x_N)_+^\mu \ \ \text{ for } x\in\R^N.
\]
Since
\[
\I_N^- z(x)\leq \sum_{i=1}^N \I_{e_i}z(x)=\I_{e_N} z(x) \ \ \text{ for } x\in \R_+^N,
\]
by Proposition \ref{0801prop1} and Remark \ref{T49-5} we have 
\[
\I_N^- z(x)+\beta z^{1-\frac{2s}\mu}(x)\leq 0 \ \ \text{ for } x\in\R_+^N,
\]
where $\beta=-c_{s,\mu}C_s$ is a positive constant. 

\smallskip

We set 
\[
w(x)=\min\{\phi(x),\,z(x)\} \ \ \text{ for } x\in\R^N. 
\]
It is obvious that $w\in C(\R^N)\cap L^\infty(\R^N)$ and that 
\[
w(x)=0 \ \ \text{ for } x\in\R^N\backslash\R_+^N \ \ \text{ and } \ \ w(x)>0 \ \ \text{ for } x\in\R_+^N.  
\]
Moreover the function $w$ is bounded from above. Indeed, if $x\in\R_+^N$, then
\begin{equation}\label{030124eq1}
w(x)\leq\phi(x)={\left|x+Re_N\right|}^{-\gamma}\leq R^{-\gamma}\leq1.
\end{equation}
Let $x\in\R_+^N$. Assume that $w(x)=z(x)$. We have 
\[
\I_N^- w(x)+\beta w^p(x)\leq \I^-_N z(x)+\beta z^{1-\frac{2s}\mu}(x)z^{p-1+\frac {2s}\mu}(x)
\]
and, since $p>1$ and hence $z^{p-1+\frac{2s}\mu}(x)\leq 1$ by \eqref{030124eq1}, 
\[
\I_N^- w(x)+\beta w^p(x)\leq \I_N^- z(x)+\beta z^{1-\frac{2s}\mu}(x)\leq 0.
\]
Assume next that $w(x)=\phi(x)$. We have 
\[
\I_N^- w(x)+\alpha w^p(x)\leq \I_N^- \phi(x)+\alpha \phi^p(x)\leq 0.
\]
Hence, we have 
\[
\I_N^- w(x)+\min\{\alpha,\beta\} w^p(x)\leq 0 \ \ \text{ for } x\in\R_+^N. 
\]
Let $\varepsilon>0$ be a constant to be fixed. Observe that for $x\in\R_+^N$,
\[
\I_N^- (\varepsilon w)(x)+(\varepsilon w)^p(x)
=\varepsilon \I_N^- w(x)+\varepsilon^p w^p(x)
\leq \varepsilon(\varepsilon^{p-1} -\min\{\alpha,\beta\})w^p(x).
\]
Noting that  $p>1$, we fix $\varepsilon>0$ so that 
\[
\varepsilon^{p-1}-\min\{\alpha,\,\beta\}\leq 0,
\]
and conclude that 
\[
\I_N^- (\varepsilon w)(x)+(\varepsilon w)^p(x)\leq 0 \ \ \text{ for } x\in \R_+^N.
\]

\end{proof}

\subsection{Theorem \ref{th-singular}}

\begin{proof}[Proof of Theorem \ref{th-singular}]
Let $u(x)=M(x_N)_+^\mu$ with $M>0$ and $\mu\in(0,s)$ to be fixed. 

\smallskip
We first consider the case $\I=\I^-_k$ for any $1\leq k\leq N$. 
Using the minimality of $\I$ among $k$-dimensional orthonormal subsets of $\R^N$, we have
\begin{equation}\label{1109peq1}
\I u(x)\leq\sum_{i=N-k+1}^N\I_{e_i}u(x)\qquad\forall x\in\R^N_+.
\end{equation}
Moreover, since $u$ is constant in all directions that are orthogonal to $e_N$, we get (whenever $k\geq2$)
\begin{equation}\label{1109peq2}
\I_{e_i}u(x)=0\qquad\forall i=N-k+1,\ldots,N-1.
\end{equation}
By \eqref{1109peq1}-\eqref{1109peq2} it follows that
\begin{equation*}
\I u(x)\leq\I_{e_N}u(x) \qquad\forall x\in\R^N_+.
\end{equation*}
In view of Proposition \ref{0801prop1} and Remark \ref{T49-5} we infer that there exists $C_{s,\mu}$ such that
\begin{equation}\label{129eq2}
\begin{split}
\I u(x)+u^p(x)&\leq M C_{s,\mu} x_N^{\mu-2s}+M^p x_N^{\mu p}\\
&=M  x_N^{\mu-2s}\left(C_{s,\mu}+M^{p-1} x_N^{\mu (p-1)+2s}\right)\qquad\forall x\in\R^N_+.
\end{split}
\end{equation}
Choosing $\mu=\frac{2s}{1-p}\in(0,s)$,  it turns out that $C_{s,\mu}<0$ and 
\begin{equation}\label{129eq3}
\I u(x)+u^p(x)\leq M  x_N^{\mu-2s}\left(C_{s,\mu}+M^{p-1}\right)=0 \qquad\forall x\in\R^N_+
\end{equation}
by taking $M={\left(\frac{1}{|C_{s,\mu}|}\right)}^\frac{1}{1-p}$. 

\smallskip
Consider now the case $\I=\I^+_N$. For any orhonormal basis $\left\{\xi_i\right\}_{i=1}^N$ one has
$$
\sum_{i=1}^N{\left\langle \xi_i,e_N\right\rangle}^2=1.
$$
Then, denoting by $\bar i\in\left\{1,\ldots,N\right\}$ the index such that 
$$
{\left\langle \xi_{\bar i},e_N\right\rangle}^2=\max_{i=1,\ldots,N}{\left\langle \xi_i,e_N\right\rangle}^2,
$$
it turns out that 
\begin{equation}\label{129eq1}
\left|\left\langle \xi_{\bar i},e_N\right\rangle\right|\geq\frac{1}{\sqrt{N}}.
\end{equation}
Again by Proposition \ref{0801prop1}, we infer that 
$$
\sum_{i=1}^N\I_{\xi_i}u(x)\leq \I_{\xi_{\bar i}}u(x)=MC_{s,\mu}\left|\left\langle \xi_{\bar i},e_N\right\rangle\right|^{2s}x_N^{\mu-2s} \qquad\forall x\in\R^N_+
$$
with $C_{s,\mu}<0$. Thus, by \eqref{129eq1},
$$
\sum_{i=1}^N\I_{\xi_i}u(x)\leq \frac{MC_{s,\mu}}{{N}^s}x_N^{\mu-2s} \qquad\forall x\in\R^N_+.
$$ 
Since the right-hand side of the above inequality does not depend on the particular choice of $\left\{\xi_i\right\}_{i=1}^N$, we have
$$
\I^+_Nu(x)\leq \frac{MC_{s,\mu}}{N^s}x_N^{\mu-2s} \qquad\forall x\in\R^N_+.
$$
Then the proof of the desired inequality  $$\I^+_Nu(x)+u^p(x)\leq0\qquad\forall x\in\R^N_+$$ follows as in \eqref{129eq2}-\eqref{129eq3} choosing $\mu=\frac{2s}{1-p}$ and $M={\left(\frac{{N}^s}{|C_{s,\mu}|}\right)}^\frac{1}{1-p}$.
\end{proof}
\begin{remark}
{\rm  
It is worth pointing out that, as far as the operators $\I^+_k$ with $k<N$ are concerned, the functions $v(x)=(x_N)_+^\mu$ are solutions of 
\begin{equation}\label{139eq1}
\I^+_ku(x)=0\;\quad\text{in $\R^N_+$}
\end{equation}
for any $\mu\in(0,s]$. Indeed, by Proposition \ref{0801prop1}, for any $k$-dimensional orthonormal set $\left\{\xi_i\right\}_{i=1}^k$ it holds
\begin{equation*}
\sum_{i=1}^k\I_{\xi_i}u(x)\leq0\;\quad\text{in $\R^N_+$}
\end{equation*}
and, with the particular choice of the first $k$ vectors of the canonical basis, we also have
\begin{equation}\label{129eq5}
\sum_{i=1}^k\I_{e_i}u(x)=0\;\quad\text{in $\R^N_+$.}
\end{equation}
 Hence \eqref{139eq1} holds and, differently from  \eqref{129eq3}, we deduce that
$$
\I^+_ku(x)+u^p(x)>0\;\quad\text{in $\R^N_+$.}
$$ 
}
\end{remark}

\appendix
\section{Appendix}\label{A}

\begin{proposition}\label{PropS}
Let either $1<p<q$ or $q<p<1$. If $u\in LSC(\R^N_+)\cap\mathcal S$ is a positive viscosity supersolution of 
\begin{equation}\label{9924eq2}
\I u(x) +u^p(x)=0\;\quad\text{in $\R^N_+$}, 
\end{equation}
where $\I$ denotes one of the operators $\I^\pm_k$, then  the function
\begin{equation*}
v(x)={\left(\frac{p-1}{q-1}\right)}^\frac{1}{q-1}u^{\frac{p-1}{q-1}}(x)
\end{equation*}
is a positive viscosity supersolution of 
\begin{equation}\label{9924eq3}
\I v(x) +v^q(x)=0\;\quad\text{in $\R^N_+$}.
\end{equation}
\end{proposition}
\begin{proof}
Let 
\begin{equation}\label{1111eq1}
\beta=\frac{p-1}{q-1}=\frac{1-p}{1-q}\in(0,1)\,,\quad\alpha=\beta^\frac{1}{q-1}.
\end{equation}
In this way
$$
v(x)=\alpha u^\beta(x)
$$
with
\begin{equation}\label{109eq1}
\beta-1+p=\beta q\;,\quad\,\alpha\beta=\alpha^q.
\end{equation}
In order prove that $v$ is a viscosity supersolution of \eqref{9924eq3}, let $\varphi\in C^2(\overline{B_\delta(x_0)})$, with $\delta>0$, touching $v$ from below at $x_0\in\R^N_+$. Then if we define
$$
\psi=\begin{cases}
\varphi & \text{in $B_\delta(x_0)$}\\
v & \text{in $\R^N\backslash B_\delta(x_0)$}
\end{cases}
$$
 we have to show that 
\begin{equation}\label{109eq2}
\I\psi(x_0)+\psi^q(x_0)\leq0.
\end{equation}
Note that $\I\psi(x_0)$ is nonincreasing with respect the parameter $\delta$ in the definition of $\psi$. Hence, without loss of generality, we may assume  $\delta$ sufficiently small such that $\psi(x)>0$ for any $x\in\R^N$.
 
The key point to obtain \eqref{109eq2} is the following inequality: if $\xi$ is any unit vector of $\R^N$, then
\begin{equation}\label{109eq3}
\I_\xi\psi(x_0)\leq\beta\psi^{\frac{\beta-1}{\beta}}(x_0)\I_\xi\psi^\frac1\beta(x_0).
\end{equation}
We first show that \eqref{109eq2} is a consequence of the above inequality, then we shall prove \eqref{109eq3}.\\
 Take an orthonormal set $\left\{\xi_i\right\}_{i=1}^k$. Using \eqref{109eq3} with $\xi$ replaced by $\xi_i$ and summing-up in $i=1,\ldots,k$ we obtain
$$
\sum_{i=1}^k\I_{\xi_i}\psi(x_0)\leq\beta\psi^{\frac{\beta-1}{\beta}}(x_0)\sum_{i=1}^k\I_{\xi_i}\psi^\frac1\beta(x_0)=\beta{\left(\alpha\psi^{\beta-1}(x_0)\right)}^{\frac1\beta}\sum_{i=1}^k\I_{\xi_i}\left(\frac\psi\alpha\right)^\frac1\beta(x_0).
$$
Then if we take the supremum among all possible $k$-orthonormal sets in the case $\I=\I^+_k$ or the infimum if $\I=\I^-_k$, we infer that
\begin{equation}\label{109eq4}
\I\psi(x_0)\leq\beta{\left(\alpha\psi^{\beta-1}(x_0)\right)}^{\frac1\beta}\I\left(\frac\psi\alpha\right)^\frac1\beta(x_0).
\end{equation}
Now, since
$$
\left(\frac\psi\alpha\right)^\frac1\beta=\begin{cases}
\left(\frac\varphi\alpha\right)^\frac1\beta & \text{in $B_\delta(x_0)$}\\
u & \text{in $\R^N\backslash B_\delta(x_0)$}
\end{cases}
$$
is a test function for $u$ at $x_0$, by the supersolution property for $u$ we obtain that
\begin{equation}\label{109eq5}
\I\left(\frac\psi\alpha\right)^\frac1\beta(x_0)\leq-u^p(x_0).
\end{equation}
Using \eqref{109eq4}-\eqref{109eq5} and the fact that
$$
\psi(x_0)=v(x_0)=\alpha u^\beta(x_0),
$$
we get 
\begin{equation}\label{109eq6}
\I\psi(x_0)\leq-\alpha\beta u^{\beta-1+p}(x_0).
\end{equation}
In view of \eqref{109eq1}, the equality 
\begin{equation}\label{109eq7}
 \alpha\beta u^{\beta-1+p}(x_0)=\psi^q(x_0)
\end{equation} 
holds and then the desired inequality \eqref{109eq2} follows from \eqref{109eq6}-\eqref{109eq7}.

\smallskip
To complete the proof it remains to show \eqref{109eq3}. For this we use the following elementary inequality (which is a straightforward application of the mean value theorem to the function $f(t)=t^\frac1\beta)$: for any positive number $a$, $b$  and $\beta\in(0,1)$ it holds 
\begin{equation}\label{109eq8}
b-a\leq\beta\, a^\frac{\beta-1}{\beta}\left(b^\frac1\beta-a^\frac1\beta\right).
\end{equation}
Using \eqref{109eq8} with $a=\psi(x_0)$ and $b=\psi(x_0\pm\tau\xi)$ we have that for any $\tau>0$
$$
\psi(x_0+\tau\xi)+\psi(x_0-\tau\xi)-2\psi(x_0)\leq\beta\psi^{\frac{\beta-1}{\beta}}(x_0)\left(\psi^\frac1\beta(x_0+\tau\xi)+\psi^\frac1\beta(x_0-\tau\xi)-2\psi^\frac1\beta(x_0)\right).
$$
Then \eqref{109eq3} easily follows multiplying both side of the previous inequality by $\tau^{-(1+2s)}$ and integrating.
\end{proof}

In the case  $p,q>0$,  Proposition \ref{PropS} holds true also in the class of nonnegative functions. 

\begin{proposition}\label{PropS2}
Let either $1<p<q$ or $0<q<p<1$. If $u\in LSC(\R^N_+)\cap\mathcal S$ is a nonnegative viscosity supersolution of 
\begin{equation}\label{99924eq2}
\I u(x) +u^p(x)=0\;\quad\text{in $\R^N_+$}, 
\end{equation}
where $\I$ denotes one of the operators $\I^\pm_k$, then  the function
\begin{equation*}
v(x)={\left(\frac{p-1}{q-1}\right)}^\frac{1}{q-1}u^{\frac{p-1}{q-1}}(x)
\end{equation*}
is a nonnegative viscosity supersolution of 
\begin{equation}\label{99924eq3}
\I v(x) +v^q(x)=0\;\quad\text{in $\R^N_+$}.
\end{equation}
\end{proposition}
\begin{proof}
Fix any $\varepsilon>0$ and note that $u_\varepsilon(x):=u(x)+\varepsilon$ is a positive supersolution of 
$$
\I u_\varepsilon(x) +u_\varepsilon^p(x)=f_\varepsilon(x)\;\quad\text{in $\R^N_+$},
$$
where $$
f_\varepsilon(x)=(u+\varepsilon)^p(x)-u^p(x).
$$
As $\varepsilon\to0^+$, then $f_\varepsilon\to0$ locally uniformly in $\R^N_+$. 

Set $v_\varepsilon(x)=\alpha u_\varepsilon^\beta(x)$ with $\alpha, \beta$ defined in \eqref{1111eq1}-\eqref{109eq1}. Note that 
\begin{equation}\label{1111eq4}
v_\varepsilon>0\quad\text{in $\R^N$ \;and\;}\quad \lim_{\varepsilon\to0^+}v_\varepsilon(x)=v(x)\quad\text{for $x\in\R^N$}.
\end{equation} 
With the same argument used in Proposition \ref{PropS}, it turns out that $v_\varepsilon$ is a viscosity supersolution of 
\begin{equation}\label{1111eq2}
\I v_\varepsilon(x)+v_\varepsilon^q(x)=g_\varepsilon(x)\qquad\text{in $\R^N_+$,}
\end{equation}
where 
\begin{equation}\label{1111eq3}
\begin{split}
g_\varepsilon(x):&=\alpha\beta{(u+\varepsilon)}^{\beta-1}(x)f_\varepsilon(x)\\
&=\alpha\beta\left({(u+\varepsilon)}^{p+\beta-1}(x)-({(u+\varepsilon)}^{\beta-1}u^p)(x)\right).
\end{split}
\end{equation}
Note that, as $\varepsilon\to0^+$, $g_\varepsilon\to0$ locally uniformly in $\R^N_+$. To check this observe that for any $0<\varepsilon<1$ and $0<\lambda<\Lambda$, if $\lambda\leq u(x)\leq\Lambda$, then for some constant $C=C(\lambda,\Lambda,p)>0$ (the Lipschitz constant of $r\mapsto r^p$ on $[\lambda,\Lambda+1]$)
$$
0\leq g_\varepsilon(x)\leq\alpha\beta\lambda^{\beta-1}C\varepsilon,
$$
and if $u(x)\leq\lambda$, then
$$
0\leq g_\varepsilon(x)\leq\alpha\beta{(\lambda+\varepsilon)}^{\beta q}.
$$
Choose any compact $K\subset\R^N_+$. Since $u$ is locally bounded in $\R^N_+$, there is a constant $\Lambda>0$ such that $u\leq\Lambda$ on $K$. Fix any $\gamma>0$. Choose $\lambda\in(0,\Lambda)$ so that
$$
\alpha\beta\lambda^{\beta q}<\gamma.
$$
Then choose $\delta\in(0,1)$ so that 
$$
\alpha\beta{(\lambda+\delta)}^{\beta q}<\gamma\quad\text{and}\quad \alpha\beta\lambda^{\beta-1}C\delta<\gamma.
$$
If $0<\varepsilon<\delta$ and $x\in K$, then
$$
g_\varepsilon(x)\leq\alpha\beta\max\left\{\lambda^{\beta-1}C\varepsilon,{(\lambda+\varepsilon)}^{\beta q}\right\}\leq\alpha\beta\max\left\{\lambda^{\beta-1}C\delta,{(\lambda+\delta)}^{\beta q}\right\}<\gamma.
$$
This show that, as $\varepsilon\to0^+$, $g_\varepsilon\to0$ uniformly on $K$. The stability of the viscosity property under the local uniform convergence, see next Lemma \ref{stability}, gives the required conclusion.
\end{proof}

\begin{lemma}\label{stability}Let $q>0$ and $v\in LSC(\R^N_+)\cap\mathcal S$. Let $\I$ be one of the operators $\I^\pm_k$. Assume that for any $\varepsilon>0$, $v_\varepsilon\in
LSC(\R^N)\cap \mathcal{S}$ is a viscosity supersolution of 
\[
\I v_\varepsilon(x)+ v_\varepsilon^q(x) \leq g_\varepsilon(x) \ \ \text{ in } \R^N_+, 
\]
where $g_\varepsilon \to 0$ locally uniformly in $\R^N_+$ as $\varepsilon\to 0^+$,
and that for some constant $C>0$,
\begin{equation}\label{11.13}
v_\varepsilon>v \geq -C\ \ \text{ in } \R^N \ \ \text{ and } \ \ v(x)=\lim_{\varepsilon\to 0^+}v_{\varepsilon}(x) \ \ \text{ for } x\in\R^N.
\end{equation}
Then, $v$ 
is a viscosity supersolution of 
\[
\I v(x)+v^q(x)\leq 0 \ \ \text{ in } \R^N_+.
\] 
\end{lemma}

\begin{proof} For $x\in\R^N$, 
$\xi\in\mathbb{S}^{N-1}$, and $\varphi\in C(\R^N)$, we set 
\[
\delta(\varphi,x,\xi)(\tau)=\varphi(x+\tau\xi)+\varphi(x-\tau\xi)-2\varphi(x) \ \ \text{ for } \tau\in\R.
\] 
If $\varphi\in C^2(\R^N)$ and $\sup_{\R^N}(|\varphi|+|D^2\varphi|)<\infty$, then 
\begin{equation}\label{eq.lem1.1}
|\delta(\varphi,x,\xi)(\tau)|\leq \min\{4\|\varphi\|_\infty,\|D^2\varphi\|_\infty\tau^2\}. 
\end{equation}

Now, let $\varphi$ be a smooth test function touching $v$ from below at $x_0\in\R^N_+$. 
Thanks to \eqref{11.13}, we may assume that $\inf_{\R^N\setminus B_r(x_0)}(v-\varphi)>0$ for any $r>0$ 
and that $\varphi(x)=-R$ for all $|x|\geq R$ 
and some $R>0$. 
Fix $r>0$ so that 
$\overline{B_r(x_0)}\subset \R^N_+$. Let $x_\varepsilon$ be a minimum point of $v_\varepsilon-\varphi$ on 
$\overline{B_r(x_0)}$. As usual, since 
\[
(v_\varepsilon-\varphi)(x_0)\geq (v_\varepsilon-\varphi)(x_\varepsilon)\geq(v-\varphi)(x_\varepsilon)\geq (v-\varphi)(x_0)=0,
\]
we see that $x_\varepsilon\to x_0$ (up to a subsequence) and $v_\varepsilon(x_\varepsilon)\to v(x_0)$ as $\varepsilon\to 0^+$. 
Since $\inf_{\R^N\setminus B_r(x_0)}(v-\varphi)>0$ for all $r>0$, if $\varepsilon$ is sufficiently small, then 
$x_\varepsilon$ is a global minimum point of $v_\varepsilon-\varphi$. 
By the supersolution property of $v_\varepsilon$, for $\varepsilon$ sufficiently small, we have
\[
\I\varphi(x_\varepsilon)+v_\varepsilon^q(x_\varepsilon)\leq g_\varepsilon (x_\varepsilon),  
\]
from which it follows that
\[
\limsup_{\varepsilon\to 0^+}\I\varphi(x_\varepsilon)+v^q(x_0)\leq 0.
\]

It remains to show that
\[
\I\varphi(x_0)\leq \limsup_{\varepsilon\to 0^+}\I\varphi(x_\varepsilon). 
\]
Consider first the case when $\I=\I_k^+$. Fix any 
orthonormal frame $\{\xi_i\}_{i=1}^k\in\mathcal{V}_k$. Note that
\[
\sum_{i=1}^k\I_{\xi_i}\varphi(x)=C_s\int\limits_0^{+\infty} \tau^{-1-2s}\sum_{i=1}^k \delta(\varphi,x,\xi_i)(\tau)d\tau,
\]
and that for all $\tau>0$,
\[
\tau^{-1-2s}\Big|\sum_{i=1}^k\delta(\varphi,x,\xi_i)(\tau)\Big|\leq k(4\|\varphi\|_\infty+\|D^2\varphi\|_\infty)
\min\{\tau^{-1-2s},\tau^{1-2s}\}.
\]
By the dominated convergence theorem,
\[
\lim_{\varepsilon\to 0^+}\sum_{i=1}^k \I_{\xi_i}\varphi(x_\varepsilon)=C_s \int\limits_0^{+\infty} \sum_{i=1}^k \tau^{-1-2s}\delta(\varphi,x_0,\xi_i)(\tau)d\tau
=\sum_{i=1}^k \I_{\xi_i}\varphi(x_0). 
\]
Hence, 
\[
\sum_{i=1}^k \I_{\xi_i}\varphi(x_0)\leq \liminf_{\varepsilon\to 0^+}\I_k^+\varphi(x_\varepsilon), 
\]
and moreover,
\[
\I_k^+\varphi(x_0)\leq \liminf_{\varepsilon\to 0^+}\I_k^+\varphi(x_\varepsilon)\leq \limsup_{\varepsilon\to 0^+}\I_k^+\varphi(x_\varepsilon).
\]

Next, consider the case when $\I=\I_k^-$. For each $\varepsilon>0$, we select an orthonormal 
frame $\{\xi_i^\varepsilon\}\in\mathcal{V}_k$ so that 
\[
\I_k^-\varphi(x_\varepsilon)+\varepsilon >\sum_{i=1}^k \I_{\xi_i^\varepsilon}\varphi(x_\varepsilon). 
\] 
We may assume, by passing to a subsequence if necessary, that  for $\varepsilon\to 0^+$
\[
\{\xi_i^\varepsilon\} \to \{\xi_i^0\} \ \ \text{ in } \R^{Nk}. 
\]
Again, by the dominated convergence theorem, 
\[
\lim_{\varepsilon\to 0^+}\sum_{i=1}^k \I_{\xi_i^\varepsilon}\varphi(x_\varepsilon)=\sum_{i=1}^k\I_{\xi_i^0}\varphi(x_0).
\]
Thus, by exploiting the choice of $\{\xi_i^\varepsilon\}$, we deduce that 
\[
\I_k^-\varphi(x_0)\leq \sum_{i=1}^k \I_{\xi_i^0}\varphi(x_0)\leq \liminf_{\varepsilon\to 0^+}\I_k^-\varphi(x_\varepsilon)\leq \limsup_{\varepsilon\to 0^+}\I_k^-\varphi(x_\varepsilon). \qedhere
\] 
\end{proof}

\vspace{0.5cm}
\noindent
\textbf{Acknowledgements.} G. Galise was partially supported by GNAMPA-INdAM. H. Ishii was partially supported by the JSPS KAKENHI Grant Nos. JP20K03688, JP23K20224, and JP23K20604.

\bigskip
\noindent
\textsc{G. Galise}: Dipartimento di Matematica  Guido Castelnuovo,\\
Sapienza Universit\`a di Roma, P.le Aldo Moro 2, I-00185 Roma, Italy.\\
E-mail: \texttt{galise@mat.uniroma1.it}

\medskip

\noindent
\textsc{H. Ishii}: Institute for Mathematics and Computer Science, \\
Tsuda University, 2-1-1 Tsuda, Kodaira, Tokyo 187-8577, Japan,.\\
E-mail: \texttt{hitoshi.ishii@waseda.jp}

\end{document}